\documentclass[11pt]{amsart}

\usepackage{a4wide, amsmath, amsfonts, amssymb, mathrsfs, amsthm, breqn}
\usepackage[hyperfootnotes=false]{hyperref}

\allowdisplaybreaks[2]

\numberwithin{equation}{section}

\newtheorem{theorem}{Theorem}[section]
\newtheorem{lemma}[theorem]{Lemma}
\newtheorem{corollary}[theorem]{Corollary}
\newtheorem{remark}[theorem]{Remark}
\newtheorem{proposition}[theorem]{Proposition}
\newtheorem{assumption}[theorem]{Assumption}
\newtheorem{remark*}{Remark}

\newcommand{\1}{\mathbf{1}}
\newcommand{\dd}{\,\mathrm{d}}
\renewcommand{\d}{\mathrm{d}}
\newcommand{\D}{\mathrm{D}}
\newcommand{\E}{\mathbb{E}}
\newcommand{\R}{\mathbb{R}}
\newcommand{\N}{\mathbb{N}}
\renewcommand{\P}{\mathbb{P}}
\newcommand{\X}{\mathbb{X}}
\newcommand{\bX}{\mathbf{X}}
\newcommand{\cD}{\mathcal{D}}
\newcommand{\cL}{\mathcal{L}}
\newcommand{\cV}{\mathcal{V}}
\newcommand{\tf}{\tilde{f}}
\newcommand{\tF}{\widetilde{F}}
\newcommand{\tX}{\widetilde{X}}
\newcommand{\tbX}{\widetilde{\bX}}
\newcommand{\tbbX}{\widetilde{\X}}
\newcommand{\ty}{\tilde{y}}
\newcommand{\tY}{\widetilde{Y}}
\newcommand{\C}{C_{F,K,X,\tX}}
\newcommand{\p}{\frac{p}{2}}
\renewcommand{\epsilon}{\varepsilon}

\title[Rough functional differential equations]{Functional differential equations \\ driven by c{\`a}dl{\`a}g rough paths}

\author[Kwossek]{Anna P. Kwossek}
\address{Anna P. Kwossek, University of Vienna, Austria}
\email{anna.paula.kwossek@univie.ac.at}

\author[Neuenkirch]{Andreas Neuenkirch}
\address{Andreas Neuenkirch, University of Mannheim, Germany}
\email{neuenkirch@uni-mannheim.de}

\author[Pr{\"o}mel]{David J. Pr{\"o}mel}
\address{David J. Pr{\"o}mel, University of Mannheim, Germany}
\email{proemel@uni-mannheim.de}

\date{\today}

\begin{document}
	
	\begin{abstract}
		The existence of unique solutions is established for rough differential equations (RDEs) with path-dependent coefficients and driven by c{\`a}dl{\`a}g rough paths. Moreover, it is shown that the associated solution map, also known as It{\^o}--Lyons map, is locally Lipschitz continuous. These results are then applied to various classes of rough differential equations, such as controlled RDEs and RDEs with delay, as well as to stochastic differential equations with delay. To that end, a joint rough path is constructed for a c{\`a}dl{\`a}g martingale and its delayed version, that corresponds to stochastic It{\^o} integration.
	\end{abstract}
	
	\maketitle
	
	
	
	
	
	\section{Introduction}
	
	Stochastic functional differential equations, also known as stochastic delay differential equations, are a natural generalization of stochastic ordinary differential equations, allowing for path-dependent coefficients which may depend on past values of the generated random dynamics. Since numerous real-world phenomena show evidence of a dependence on the past as well as a stochastic behaviour, stochastic functional differential equations serve as mathematical models in many areas ranging from biology to finance. For classical introductions to stochastic functional differential equations we refer, e.g., to~\cite{Mohammed1984, Mohammed1998}.
	
	A deterministic approach to stochastic differential equations is provided by rough path theory, initiated by Lyons~\cite{Lyons1998}. Originally designed to treat stochastic ordinary differential equations, it has been extended in various directions, for instance, allowing to deal with stochastic Volterra equations~\cite{Deya2009}, reflected stochastic differential equations~\cite{Aida2015}, stochastic inclusion equations~\cite{Bailleul2021}, and different classes of stochastic partial differential equations~\cite{Hairer2011, Caruana2011}. These rough path approaches contributed many novel insights to the study of the aforementioned equations, such as, but not limited to, new well-posedness and stability results. Comprehensive introductions to rough path theory can be found, e.g., in~\cite{Lyons2007, Friz2020}.
	
	In the present paper, we study \emph{rough functional differential equations} (RFDEs)
	\begin{equation}\label{eq: intro RFDE}
		Y_t = y_t +\int_0^t b_s(Y) \dd s + \int_0^t \sigma_s(Y) \dd \bX_s, \quad t \in [0,T],
	\end{equation}
	where the driving signal~$\bX$ is a c{\`a}dl{\`a}g $p$-rough path for $p \in (2,3)$, the initial condition~$y$ is a given controlled path, and the coefficients~$b, \sigma$ are non-anticipative functionals, mapping a controlled path to a controlled path. Assuming a quadratic growth and a Lipschitz-type condition on the path-dependent coefficients~$b, \sigma$, which both are formulated on the space of controlled paths, we establish the existence of a unique solution to the RFDE~\eqref{eq: RFDE}. To that end, we rely on the theory of c{\`a}dl{\`a}g rough paths, as introduced by Friz, Shekhar and Zhang~\cite{Friz2017,Friz2018}, as well as Banach's fixed point theorem. Moreover, we show that the solution map, also known as It{\^o}--Lyons map, mapping the input (initial condition, coefficients, driving signal) of an RFDE to its solution, is locally Lipschitz continuous with respect to suitable distances on the associated spaces of coefficients, controlled paths and rough paths. Let us remark that the continuity of the It{\^o}--Lyons map is one of the most fundamental insights of rough path theory, with many applications to stochastic differential equations, cf.~e.g.~\cite{Friz2020}.
	
	The presented results on rough functional differential equations provide a unifying theory, recovering and extending various previous results on different classes of rough differential equations with path-dependent coefficients. Indeed, we deduce the existence of unique solutions as well as the local Lipschitz continuity of the It{\^o}--Lyons map for classical rough differential equations (RDEs), controlled RDEs, RDEs with discrete time dependence and RDEs with constant/variable delay, that are all driven by c{\`a}dl{\`a}g $p$-rough paths for $p \in (2,3)$.
	
	In the existing literature, there are several different approaches to deal with rough functional differential equations driven by continuous rough paths. Since the theory of (continuous) rough paths works nicely for infinite dimensional Banach spaces, RDEs with path-dependent coefficients can be treated as Banach space-valued RDEs, see e.g.~\cite{Bailleul2014}, which requires the coefficients to be Fr{\'e}chet differentiable and, thus, excludes some interesting examples. Existence, uniqueness and stability results are established by Neuenkirch, Nourdin and Tindel~\cite{Neuenkirch2008} for RDEs with constant delay. The existence of a solution is proven by Ananova~\cite{Ananova2023} for RDEs with path-dependent coefficients, which are assumed to be Dupire differentiable~\cite{Dupire2019}, and by Aida~\cite{Aida2024} for RDEs with coefficients containing path-dependent bounded variation terms. The latter two approaches rely on Schauder's fixed point theorem. Another exemplary class of RFDEs are reflected rough differential equations, see e.g.~\cite{Aida2015,Deya2019}, which, in general, do not possess a unique solution, see~\cite{Gassiat2021}.
	
	Most applications of rough path theory to stochastic differential equations (SDEs) crucially rely on the construction of suitable (random) rough paths. To apply the developed theory on RFDEs to It{\^o} SDEs with constant delay, we show that a c{\`a}dl{\`a}g martingale together with its delayed version can be lifted to a random rough path in the spirit of stochastic It{\^o} integration. The key challenge to obtain the ``delayed'' rough path is that a martingale together with its delayed version is, in general, not a martingale itself, thus preventing the direct use of stochastic It{\^o} integration. For related constructions of random rough paths above fractional Brownian motions we refer to~\cite{Neuenkirch2008, Tindel2012, Besalu2014, Chhaibi2020}. Consequently, one can apply the continuity of the It{\^o}--Lyons map to derive pathwise stability results for stochastic differential equations with constant delay, which plays an important role in many applications, see e.g.~\cite{Banos2018}. In particular, the map $y \mapsto Y$, mapping the initial condition~$y$ to the associated solution~$Y$ of an SDE with constant delay, is continuous on the space of controlled paths. This resolves an old observation, pointed out by Mohammed~\cite{Mohammed1986}, about the non-continuity of the flow of stochastic differential equations with delay, for which the initial condition is in fact an initial path.
	
	\medskip
	
	\noindent \textbf{Organization of the paper:} In Section~\ref{sec: RFDE} we provide existence, uniqueness and continuity results for rough functional differential equations. In Section~\ref{sec: examples} we prove that various classes of rough differential equations are covered by the presented results on rough functional differential equations. In Section~\ref{sec: rough path lift} we establish the existence of the It{\^o} rough path lift of delayed martingales and discuss applications to stochastic differential equations with delay. Appendix~\ref{sec: appendix} contains some auxiliary estimates for rough integrals.
	
	\medskip
	
	\noindent \textbf{Acknowledgments:} A.~P.~Kwossek and D.~J.~Pr{\"o}mel gratefully acknowledge financial support by the Baden-W{\"u}rttemberg Stiftung. A.~P.~Kwossek was affiliated with the University of Mannheim for the majority of this project's duration.
	
	\section{Existence, uniqueness and continuity}\label{sec: RFDE}
	
	Before treating rough functional differential equations (RFDEs), we recall the necessary definitions and some essentials from the theory of c{\`a}dl{\`a}g rough paths, as introduced by Friz and Shekhar~\cite{Friz2017} and Friz and Zhang~\cite{Friz2018}.
	The theory of c{\`a}dl{\`a}g rough paths extends the classical rough path theory, allowing to deal with many stochastic processes with jumps \cite{Chevyrev2018}, and has numerous applications, e.g., in probability theory~\cite{Friz2023}, numerical analysis~\cite{Friz2018} and mathematical finance~\cite{Allan2024}.
	
	\subsection{Essentials on rough path theory}
	
	Throughout, let $T > 0$ be a fixed finite time horizon. Let $\Delta_T := \{(s,t) \in [0,T]^2 \, : \, s \leq t\}$ be the standard $2$-simplex. A function $w \colon \Delta_T \to [0,\infty)$ is called a \emph{control function} if it is superadditive, in the sense that $w(s,u) + w(u,t) \leq w(s,t)$ for all $0 \leq s \leq u \leq t \leq T$. We write $w(s,t-) := \lim_{u \uparrow t} w(s,u)$ if $s < t$, and $w(s,t-) := 0$ if $s = t$.
	
	Whenever $(B,\|\cdot\|)$ is a normed space and $f, g \colon B \to \R$ are two functions on $B$, we shall write $f \lesssim g$ or $f \leq C g$ to mean that there exists a constant $C > 0$ such that $f(x) \leq C g(x)$ for all $x \in B$. The constant $C$ may depend on the normed space, e.g.~through its dimension or regularity parameters, and, if we want to emphasize the dependence of the constant $C$ on some particular variables, say $\alpha_1, \ldots, \alpha_n$, then we will write $\lesssim_{\alpha_1, \ldots, \alpha_n}$ or $C = C_{\alpha_1, \ldots, \alpha_n}$.
	Unless otherwise stated, the dependence of the implicit constants on the variables is locally bounded; that is, if $\alpha_1 \in A_1, \ldots, \alpha_n \in A_n$, where $A_1, \ldots, A_n$ are compact subsets of the range of $\alpha_1, \ldots, \alpha_n$ respectively, then we have that $\sup_{\alpha_1 \in A_1, \ldots, \alpha_n \in A_n} C_{\alpha_1, \ldots, \alpha_n} < \infty$.
	
	\medskip
	
	For two vector spaces, the space of linear maps from $E_1 \to E_2$ is denoted by $\cL(E_1; E_2)$; and we write $C^k_b = C^k_b(\R^m;\cL(\R^d;\R^n))$ for the space of $k$-times differentiable (in the Fr{\'e}chet sense) functions $f \colon \R^m \to \cL(\R^d;\R^n)$ such that $f$ and all its derivatives up to order $k$ are continuous and bounded. We equip this space with the norm
	\begin{equation*}
		\|f\|_{C^k_b} := \|f\|_\infty + \|\D f\|_\infty + \cdots + \|\D^k f\|_\infty,
	\end{equation*}
	where $\D^r f$ denotes the $r$-th order derivative of $f$, and $\|\hspace{0.5pt} \cdot \hspace{0.5pt}\|_{\infty}$ denotes the supremum norm on the corresponding spaces of operators.
	
	\medskip
	
	For a normed space $(E,|\cdot|)$, let $D([0,T];E)$ be the set of c{\`a}dl{\`a}g (right-continuous with left-limits) paths from $[0,T]$ to $E$. For $p \geq 1$, the $p$-variation of a path $X \in D([0,T];E)$ is given by
	\begin{equation*}
		\|X\|_p := \|X\|_{p,[0,T]} \qquad \text{with} \qquad \|X\|_{p,[s,t]} := \bigg(\sup_{\mathcal{P}\subset[s,t]} \sum_{[u,v]\in \mathcal{P}} |X_v - X_u|^p \bigg)^{\frac{1}{p}}, \quad (s,t) \in \Delta_T,
	\end{equation*}
	where the supremum is taken over all possible partitions $\mathcal{P}$ of the interval $[s,t]$. We recall that, given a path $X$, we have that $\|X\|_p < \infty$ if and only if there exists a control function $w$ such that\footnote{Here and throughout, we adopt the convention that $\frac{0}{0} := 0$.}
	\begin{equation*}
		\sup_{(u,v) \in \Delta_T} \frac{|X_v - X_u|^p}{w(u,v)} < \infty.
	\end{equation*}
	We write $D^p = D^p([0,T];E)$ for the space of paths $X \in D([0,T];E)$ which satisfy $\|X\|_p < \infty$. Moreover, for a path $X \in D([0,T];\R^d)$, we will often use the shorthand notation:
	\begin{equation*}
		X_{s,t} := X_t - X_s \quad \text{and} \quad X_{t-}:= \lim_{u\uparrow t} X_u, \qquad \text{for} \quad (s,t) \in \Delta_T.
	\end{equation*}
	
	For $p > 2$ and a two-parameter function $\X \colon \Delta_T \to E$, we similarly define
	\begin{equation*}
		\|\X\|_\p := \|\X\|_{\p,[0,T]} \qquad \text{with} \qquad \|\X\|_{\p,[s,t]} := \bigg(\sup_{\mathcal{P} \subset [s,t]} \sum_{[u,v] \in \mathcal{P}} |\X_{u,v}|^\p\bigg)^{\frac{2}{p}}, \quad (s,t) \in \Delta_T.
	\end{equation*}
	We write $D_2^\p = D_2^\p(\Delta_T;E)$ for the space of all functions $\X \colon \Delta_T \to E$ which satisfy $\|\X\|_\p < \infty$, and are such that the maps $s \mapsto \X_{s,t}$ for fixed $t$, and $t \mapsto \X_{s,t}$ for fixed $s$, are both c{\`a}dl{\`a}g.
	
	\medskip
	
	For $p \in (2,3)$, a pair $\bX = (X,\X)$ is called a \emph{c{\`a}dl{\`a}g $p$-rough path} over $\R^d$ if
	\begin{enumerate}
		\item[(i)] $X \in D^p([0,T];\R^d)$ and $\X \in D_2^{\p}(\Delta_T;\R^{d \times d})$, and
		\item[(ii)] Chen's relation: $\X_{s,t} = \X_{s,u} + \X_{u,t} + X_{s,u} \otimes X_{u,t}$ holds for all $0 \leq s \leq u \leq t \leq T$,
	\end{enumerate}
	where $\otimes$ denotes the usual tensor product. In component form then, condition~(ii) states that $\X^{ij}_{s,t} = \X^{ij}_{s,u} + \X^{ij}_{u,t} + X^i_{s,u} X^j_{u,t}$ for every $i$ and $j$. We will denote the space of c{\`a}dl{\`a}g $p$-rough paths by $\cD^p = \cD^p([0,T];\R^d)$. On the space $\cD^p([0,T];\R^d)$, we use the natural seminorm
	\begin{equation*}
		\|\bX\|_p := \|\bX\|_{p,[0,T]} \qquad \text{with} \qquad \|\bX\|_{p,[s,t]} := \|X\|_{p,[s,t]} + \|\X\|_{\p,[s,t]}
	\end{equation*}
	for $(s,t) \in \Delta_T$, and the induced distance
	\begin{equation*}
		\|\bX;\tbX\|_p := \|\bX;\tbX\|_{p,[0,T]} \qquad \text{with} \qquad
		\|\bX;\tbX\|_{p,[s,t]} := \|X - \tX\|_{p,[s,t]} + \|\X - \tbbX\|_{\p,[s,t]},
	\end{equation*}
	for $(s,t) \in \Delta_T$.
	
	\medskip
	
	Let $p \in (2,3)$, and $X \in D^p([0,T];\R^d)$. We say that a pair $(Y,Y')$ is a \emph{controlled path} (with respect to $X$), if
	\begin{equation*}
		Y \in D^p([0,T];E), \quad Y' \in D^p([0,T];\cL(\R^d;E)), \quad \text{and} \quad R^Y \in D^\p_2(\Delta_T;E),
	\end{equation*}
	where $R^Y$ is defined by
	\begin{equation*}
		Y_{s,t} = Y'_s X_{s,t} + R^Y_{s,t} \qquad \text{for all} \quad (s,t) \in \Delta_T.
	\end{equation*}
	We write $\cV^p_X = \cV^p_X([0,T];E)$ for the space of $E$-valued controlled paths, which becomes a Banach space when equipped with the norm $(Y,Y) \mapsto |Y_0| + \|Y,Y'\|_{X,p}$, where
	\begin{equation*}
		\|Y,Y'\|_{X,p} := \|Y,Y'\|_{X,p,[0,T]} \quad \text{with} \quad \|Y,Y'\|_{X,p,[s,t]} := |Y'_s| + \|Y'\|_{p,[s,t]} + \|R^Y\|_{\p,[s,t]}
	\end{equation*}
	for $(s,t) \in \Delta_T$. We point out that, by definition, for $(s,t) \in \Delta_T$,
	\begin{equation*}
		|Y_{s,t}| \leq |Y'_s| |X_{s,t}| + |R^Y_{s,t}| \qquad \text{and} \qquad |Y'_t| \leq |Y'_0| + |Y'_{0,t}|,
	\end{equation*}
	which implies that
	\begin{equation*}
		\|Y\|_p \leq C_p (\|Y'\|_\infty \|X\|_p + \|R^Y\|_\p)
		\qquad \text{and} \qquad
		\|Y'\|_\infty \leq |Y'_0| + \|Y'\|_p,
	\end{equation*}
	where $\|Y'\|_\infty := \sup_{t \in [0,T]} |Y'_t|$ denotes the supremum seminorm of the path $Y'$.
	
	Given $X, \tX \in D^p$, we further introduce the standard ``distance''
	\begin{equation*}
		\|Y,Y';\tY,\tY'\|_{X,\tX,p} := \|Y,Y';\tY,\tY'\|_{X,\tX,p,[0,T]}
	\end{equation*}
	with
	\begin{equation*}
		\|Y,Y';\tY,\tY'\|_{X,\tX,p,[s,t]} := |Y'_s - \tY'_s| + \|Y' - \tY'\|_{p,[s,t]} + \|R^Y - R^{\tY}\|_{\p,[s,t]}
	\end{equation*}
	for $(s,t) \in \Delta_T$, whenever $(Y,Y') \in \cV^p_X, (\tY,\tY') \in \cV^p_{\tX}$. Note that $\cV^p_X$ and $\cV^p_{\tX}$ are, in general, different Banach spaces; if $X = \tX$, we write $\|\cdot \,; \cdot\|_{X,p,[s,t]}$.
	
	\medskip
	
	Given $p \in (2,3)$, $\bX = (X,\X) \in \cD^p([0,T];\R^d)$ and $(Y,Y') \in \cV_X^p([0,T];\cL(\R^d;\R^k))$, the (forward) rough integral
	\begin{equation}\label{eq: rough path integral}
		\int_s^t Y_r \dd \bX_r :=  \lim_{|\mathcal{P}| \to 0} \sum_{[u,v] \in \mathcal{P}} (Y_u X_{u,v} + Y'_u \X_{u,v}), \qquad (s,t) \in \Delta_T,
	\end{equation}
	exists (in the classical mesh Riemann--Stieltjes sense), where the limit is taken along any sequence of partitions $(\mathcal{P}^n)_{n\in \N}$ of the interval $[s,t]$ such that $|\mathcal{P}^n| \to 0$ as $n\to \infty$. More precisely, in writing the product $Y_u X_{u,v}$, we apply the operator $Y_u \in \cL(\R^d;\R^k)$ onto $X_{u,v} \in \R^d$; and in writing the product $Y'_u \X_{u,v}$, we use the natural identification of $\cL(\R^d;\cL(\R^d;\R^k))$ with $\cL(\R^d \otimes \R^d;\R^k)$. The rough integral comes with the estimate
	\begin{equation*}
		\bigg|\int_s^t Y_r \dd \bX_r - Y_s X_{s,t} - Y'_s \X_{s,t}\bigg| \leq C \Big( \|R^Y\|_{\p,[s,t)} \|X\|_{p,[s,t]} + \|Y'\|_{p,[s,t)} \|\X\|_{\p,[s,t]} \Big)
	\end{equation*}
	for some constant $C$ depending only on $p$; see~\cite[Proposition~2.6]{Friz2018}, where
	\begin{equation*}
		\|Y'\|_{p,[s,t)} := \sup_{u < t} \|Y'\|_{p,[s,u]} \quad \text{and} \quad
		\|R^Y\|_{\p,[s,t)} := \sup_{u < t} \|R^Y\|_{\p,[s,u]}.
	\end{equation*}
	The estimate implies that $(\int_0^\cdot Y_r \dd \bX_r, Y) \in \cV_X^p([0,T];\R^k)$ is a controlled path with respect to $X$, see also~\cite[Remark~2.8]{Friz2018}.
	
	For details on the construction of the rough integral with respect to c{\`a}dl{\`a}g $p$-rough paths and its properties, we refer to~\cite{Friz2017, Friz2018}, and we provide some auxiliary estimates for the rough integral in Appendix~\ref{sec: appendix}.
	
	\medskip
	
	Let us now consider the rough functional differential equation (RFDE)
	\begin{equation}\label{eq: RFDE}
		Y_t = y_t + \int_0^t F_s(Y) \dd \bX_s, \quad t \in [0,T],
	\end{equation}
	where $\bX \in \cD^p([0,T];\R^d)$ is a c{\`a}dl{\`a}g $p$-rough path for $p \in (2,3)$, $(y,y') \in \cV^p_X([0,T];\R^k)$ is a given controlled path with respect to $X$ and further, where $(F,F') \colon \cV^p_X([0,T];\R^k) \to \cV^p_X([0,T];\cL(\R^d;\R^k))$ is a non-anticipative functional, i.e.
	\begin{itemize}
		\item[(i)] $(F_\cdot(Y),F'_\cdot(Y,Y')) \in \cV^p_X([0,T];\cL(\R^d;\R^k))$,
		\item[(ii)] $(F_t(Y),F'_t(Y,Y')) = (F_t(Y_{\cdot \wedge t}), F'_t(Y_{\cdot \wedge t}, Y'_{\cdot \wedge t}))$ for all $t \in [0,T]$,
	\end{itemize}
	for every $(Y,Y') \in \cV^p_X([0,T];\R^k)$. The integral in~\eqref{eq: RFDE} is defined as a (forward) rough integral, see~\eqref{eq: rough path integral} for its definition. Note that the RFDE~\eqref{eq: intro RFDE} can be re-written in the form of~\eqref{eq: RFDE}, using a standard time-extension of the driving rough path.
	
	\subsection{Existence and uniqueness}
	
	To prove the existence of a unique solution to the rough functional differential equation~\eqref{eq: RFDE}, we postulate a quadratic growth and a Lipschitz-type condition on the path-dependent coefficient~$(F,F')$, formulated on the associated path spaces. While a Lipschitz-type condition is expected, the quadratic growth condition appears to be natural in the presented context of (second order) controlled paths, which corresponds to a Taylor expansion with quadratic remainder term.
	
	\begin{assumption}\label{assumption: existence}
		Let $X\in D^p([0,T];\R^d)$ be given. For every $K > 0$, there exist constants $C_F > 0$, which depends on $p$ and the functional $F$, and $C_{F,K,X} > 0$, which additionally depends on $K$ and $X$, such that the map
		\begin{equation*}
			(F,F') \colon \cV^p_X([0,T];\R^k) \to \cV^p_X([0,T];\cL(\R^d;\R^k))
		\end{equation*}
		satisfy, for all $(Y,Y'), (\tY,\tY') \in \cV^p_X$, and every $0 \leq s < t \leq T$,  \\
		\textup{(i)} the growth conditions:
		\begin{align*}
			&|F_t(Y)| \leq C_F, \\
			&|F_{t-,t}(Y)| \leq C_F (1 + \|Y\|_{p,[s,t)} + |Y_{t-,t}|), \\
			&\|F(Y)\|_{p,[s,t]} \leq C_F (1 + (|Y'_s| + \|Y'\|_{p,[s,t]}) \|X\|_{p,[s,t]} + \|R^Y\|_{\p,[s,t]}),
			\quad\text{and}\\
			&\|F(Y),F'(Y,Y')\|_{X,p,[s,t]} \leq C_F (1 + \|Y,Y'\|_{X,p,[s,t]})^2 (1 + \|X\|_{p,[s,t]})^2;
		\end{align*}
		\textup{(ii)} the Lipschitz conditions:
		\begin{align*}
			&\|F(Y) - F(\tY)\|_{p,[s,t]} \leq C_{F,K,X} (|Y_s - \tY_s| + \|Y - \tY\|_{p,[s,t]}),
			\quad\text{and}\\
			&\|F(Y),F'(Y,Y');F(\tY),F'(\tY,\tY')\|_{X,p,[s,t]}\\
			&\quad \leq C_{F,K,X} (|Y_s - \tY_s| + \|Y,Y';\tY,\tY'\|_{X,p,[s,t]}),
		\end{align*}
		if $\|Y,Y'\|_{X,p,[s,t]}, \|\tY,\tY'\|_{X,p,[s,t]} \leq K$.
	\end{assumption}
	
	\begin{remark}
		The growth and Lipschitz conditions in Assumption~\ref{assumption: existence} are formulated in terms of both the $p$-variation of $Y$ and the controlled path norm $(Y,Y') \mapsto |Y_0| + \|Y,Y'\|_{X,p}$ on the space $\cV^p_X$ of controlled paths $(Y,Y^{'})$. To deduce the existence of a unique solution to the RFDE~\eqref{eq: RFDE} under a growth and Lipschitz conditions formulated only in terms of the controlled path norm seems to be far from being obvious. Moreover, notice that the common examples of RDEs with path-dependent coefficients do satisfy Assumption~\ref{assumption: existence}, see Section~\ref{sec: examples} below, demonstrating that Assumption~\ref{assumption: existence} is, indeed, a natural generalization of the conditions on the coefficients postulated in the existing literature.
	\end{remark}
	
	Based on Assumption~\ref{assumption: existence}, we obtain the following global existence and uniqueness result for rough functional differential equations.
	
	\begin{theorem}\label{thm: existence and uniqueness}
		Let $\bX \in \cD^p([0,T];\R^d)$ be a c{\`a}dl{\`a}g $p$-rough path for $p \in (2,3)$, and $(y,y') \in \cV^p_X([0,T];\R^k)$ be a given controlled path with respect to $X$. Suppose that the non-anticipative functional $(F,F') \colon \cV^p_X([0,T];\R^k) \to \cV^p_X([0,T];\cL(\R^d;\R^k))$ satisfies Assumption~\ref{assumption: existence} given~$X$. Then, there exists a unique solution to the rough functional differential equation~\eqref{eq: RFDE}, i.e. there exists a unique controlled path $(Y,Y') \in \cV^p_X([0,T];\R^k)$, with $Y' = y' + F(Y)$, such that
		\begin{equation*}
			Y_t = y_t + \int_0^t F_s(Y) \dd \bX_s, \qquad  t \in [0,T].
		\end{equation*}
		Moreover, there exists a componentwise non-decreasing function $K_p \colon [0, \infty)^3 \to [0,\infty)$ such that
		\begin{equation*}
			\|Y,Y'\|_{X,p} \leq K_p(\|y,y'\|_{X,p}, C_F , \|\bX\|_p).
		\end{equation*}
	\end{theorem}
	
	The proof relies on a fixed point approach using Banach's fixed point theorem.
	
	\begin{proof}
		\emph{Step~1: Local solution.} We may assume that
		\begin{equation*}
			\|\bX\|_p \leq 1 \quad \text{and} \quad \|y'\|_p + \|R^{y}\|_\p \leq 1.
		\end{equation*}
		For $t \in (0,T]$, we define the map $\mathcal{M}_t \colon \cV^p_X([0,t];\R^k) \to \cV^p_X([0,t];\R^k)$ by
		\begin{equation*}
			(Y,Y') \mapsto (Z,Z') := \mathcal{M}_t(Y,Y') := \bigg( y_{\cdot} + \int_0^{\cdot} F_s(Y) \dd \bX_s, y'_{\cdot} + F_{\cdot} (Y) \bigg),
		\end{equation*}
		noting that $\mathcal{M}_t(Y,Y')$ is a controlled path with respect to $X$ as $\cV^p_X$ is a Banach space, and introduce the subset of controlled paths
		\begin{equation*}
			\mathcal{B}_t := \bigg\{ (Y,Y') \in \cV^p_X([0,t];\R^k) \, : \,
			\begin{array}{l}
				(Y_0,Y_0') = (y_0, y'_0 + F_0(y)), \\ \|(Y - y)'\|_{p,[0,t]} \leq 4 C_F, \|R^Y - R^y\|_{\p,[0,t]} \leq 1
			\end{array} \bigg\},
		\end{equation*}
		which is a complete set as a closed subset of $\cV^p_X([0,t];\R^k)$, cf.~\cite[Section 3.2]{Friz2018}.
		
		\emph{Invariance.} For any $(Y,Y') \in \mathcal{B}_t$, we have that
		\begin{align*}
			&\|(Z-y)'\|_{p,[0,t]} = \|F(Y)\|_{p,[0,t]} \\
			&\quad \leq C_F (1 + (|Y'_0| + \|Y'\|_{p,[0,t]}) \|X\|_{p,[0,t]} + \|R^Y\|_{\p,[0,t]}) \\
			&\quad \leq C_F + C_F(|Y'_0| + \|Y'\|_{p,[0,t]}) \|X\|_{p,[0,t]} + C_F \|R^{Y-y}\|_{\p,[0,t]} + C_F \|R^y\|_{\p,[0,t]} \\
			&\quad \leq C_F(1 + |Y'_0| + \|Y'\|_{p,[0,t]}) \|X\|_{p,[0,t]} + 3 C_F,
		\end{align*}
		since $(F,F')$ satisfies Assumption~\ref{assumption: existence}~(i), and by the local estimate for rough integration, see Lemma~\ref{lemma: estimate rough integral},
		\begin{align*}
			&\|R^Z - R^y\|_{\p,[0,t]} = \|R^{\int_0^\cdot F(Y) \d \bX}\|_{\p,[0,t]} \\
			&\quad \lesssim C_F (1 + \|Y,Y'\|_{X,p,[0,t]})^2 (1 + \|X\|_{p,[0,t]})^2 \|\bX\|_{p,[0,t]},
		\end{align*}
		where the implicit multiplicative constant depends only on $p$.
		Hence, for $t = t_1$ sufficiently small we obtain that $\mathcal{B}_{t_1}$ is invariant under $\mathcal{M}_{t_1}$. Note that $t_1$ depends on $p$, $|y'_0|$, $C_F$ and $\|\bX\|_p$.
		
		\emph{Contraction.} Let $(Y,Y'), (\tY,\tY') \in \mathcal{B}_t$ for some $t \in (0,t_1]$, that is, setting $K:= 5 (1 + \|y,y'\|_{X,p} + C_F)$, it holds that $\|Y,Y'\|_{X,p,[0,t]}, \|\tY,\tY'\|_{X,p,[0,t]} \leq K$. We have that
		\begin{align*}
			&\|Z' - \widetilde{Z}'\|_{p,[0,t]} = \|F(Y) - F(\tY)\|_{p,[0,t]} \\
			&\quad \leq C_{F,K,X} \|Y - \tY\|_{p,[0,t]} \\
			&\quad \lesssim_p C_{F,K,X} (\|Y' - \tY'\|_{p,[0,t]} \|X\|_{p,[0,t]} + \|R^Y - R^{\tY}\|_{\p,[0,t]}).
		\end{align*}
		Further, due to Assumption~\ref{assumption: existence}~(ii) and Lemma~\ref{lemma: Lipschitz estimate rough integral}, it holds that
		\begin{align*}
			&\|R^Z - R^{\widetilde{Z}}\|_{\p,[0,t]} = \|R^{\int_0^\cdot F(Y) \d \bX - \int_0^\cdot F(\tY) \d \bX}\|_{\p,[0,t]} \\
			&\quad \lesssim C_{F,K,X} \|Y,Y';\tY,\tY'\|_{X,p,[0,t]} \|\bX\|_{p,[0,t]},
		\end{align*}
		where the implicit multiplicative constant depends on $p$ and $\|\bX\|_p$. Defining an equivalent norm on $\cV^p_X$ by
		\begin{equation*}
			\|Y,Y'\|^{(\delta)}_{X,p,[0,t]} := |Y'_0| + \|Y'\|_{p,[0,t]} + \delta \|R^Y\|_{\p,[0,t]}, \quad \text{for } \delta \geq 1,
		\end{equation*}
		we then deduce that
		\begin{align*}
			&\|Z -\widetilde{Z}, Z' - \widetilde{Z}'\|^{(\delta)}_{X,p,[0,t]} \\
			&\quad \lesssim C_{F,K,X} (\|Y' - \tY'\|_{p,[0,t]} \|X\|_{p,[0,t]} + \|R^Y - R^{\tY}\|_{\p,[0,t]}) \\
			&\qquad + \delta C_{F,K,X} (\|Y' - \tY'\|_{p,[0,t]} + \|R^Y - R^{\tY}\|_{\p,[0,t]}) \|\bX\|_{p,[0,t]} \\
			&\quad \lesssim C_{F,K,X} (1 + \delta) \|\bX\|_{p,[0,t]} \|Y' - \tY'\|_{p,[0,t]} + C_{F,K,X} (1 + \delta \|\bX\|_{p,[0,t]}) \|R^Y - R^{\tY}\|_{\p,[0,t]} \\
			&\quad \lesssim C_{F,K,X} \bigg( (1 + \delta) \|\bX\|_{p,[0,t]} \vee \frac{1 + \delta \|\bX\|_{p,[0,t]}}{\delta} \bigg) \|Y - \tY, Y' - \tY'\|^{(\delta)}_{X,p,[0,t]},
		\end{align*}
		where the implicit multiplicative constant depends on $p$ and $\|\bX\|_p$. Hence, we can choose $\delta$ sufficiently large and $t = t_2 \leq t_1$ sufficiently small such that
		\begin{equation*}
			C_{F,K,X} \bigg( (1 + \delta) \|\bX\|_{p,[0,t_2]} \vee \frac{1 + \delta \|\bX\|_{p,[0,t_2]}}{\delta} \bigg) \leq 1,
		\end{equation*}
		where the left-hand side is up to a multiplicative constant which depends on $p$ and $\|\bX\|_p$.
		It follows that $\mathcal{M}_{t_2}$ is a contraction on the subset of controlled paths $(\mathcal{B}_{t_2},\|\cdot\|^{(\delta)}_{X,p,[0,t_2]})$. Hence, by Banach's fixed point theorem, there exists a unique fixed point of the map~$\mathcal{M}_{t_2}$, which is the unique solution of the RFDE~\eqref{eq: RFDE} over the time interval $[0,t_2]$.
		
		\emph{Step~2: Global solution.} Let $w \colon \Delta_T \to [0,\infty)$ be the right-continuous control function given by
		\begin{equation*}
			w(s,t) := \|X\|^p_{p,[s,t]} + \|\X\|^\p_{\p,[s,t]}, \quad (s,t) \in \Delta_T.
		\end{equation*}
		We infer from Step~1 that there exists a constant $\gamma > 0$, which depends on $p$, $\|y,y'\|_{X,p}$, $C_F$, $C_{F,K,X}$ and $\|\bX\|_p$, such that the local solution $(Y,Y')$ established above exists on any interval $[s,t]$ such that $w(s,t) \leq \gamma$, given any initial condition $\xi \in \cV^p_X$ with $|\xi'_s| \leq \|y,y'\|_{X,p}$.
		
		By~\cite[Lemma~1.5]{Friz2018}, there exists a partition $\mathcal{P} = \{0 = \tau_0 < \tau_1 < \dots < \tau_N = T\}$ of $[0,T]$, such that $w(\tau_i,\tau_{i+1}-) < \gamma$ for every $i = 0, 1, \ldots, N-1$. We can then define the solution $(Y,Y')$ on each of the half-intervals $[\tau_i,\tau_{i+1})$. Given the solutions on $[\tau_i,\tau_{i+1})$, the values $Y_{\tau_{i+1}}$ at the right end-point of the interval are uniquely determined by the jumps of $\bX$ at time $\tau_{i+1}$. More precisely, let $y_{0;\cdot} = y_\cdot$, and define $y_i$, $i = 1, \ldots, N-1$, by
		\begin{equation*}
			y_{i;t} = y_t + Y_{\tau_i-} - y_{\tau_i-} + F_{\tau_i-}(Y) X_{\tau_i-,\tau_i} + F'_{\tau_i-}(Y,Y') \X_{\tau_i-,\tau_i}, \quad t \in [\tau_i,\tau_{i+1}).
		\end{equation*}
		We note that $|y'_{i,\tau_i}| = |y'_i| \leq \|y,y'\|_{X,p}$. Given the initial condition $(y_i,y'_i) \in \cV^p_X$, we obtain the solution $(Y,Y')$ on $[\tau_i,\tau_{i+1})$, $i = 0, 1, \ldots, N-1$. By pasting the solutions on each of these subintervals together, with $Y_T = y_{N;T}$, we obtain a unique global solution $(Y,Y')$ of the RFDE~\eqref{eq: RFDE} on the interval~$[0,T]$.
		
		\emph{Step~3: A priori estimate.} It remains to show the existence of a componentwise non-decreasing function $K_p \colon [0, \infty)^3 \to [0,\infty)$ such that
		\begin{equation*}
			\|Y,Y'\|_{X,p} \leq K_p(\|y,y'\|_{X,p}, C_F, \|\bX\|_p).
		\end{equation*}
		Analogously to Step~2, we can obtain a partition $\mathcal{P} = \{0 = \tau_0 < \tau_1 < \dots < \tau_N = T\}$ and define the solution $(Y,Y')$ on each of the half-intervals $[\tau_i,\tau_{i+1})$, $i = 0, 1, \dots, N-1$. We recall the definition of $\mathcal{B}_t$ and note that the defining estimates also hold in terms of $y$ since the $p$-variation is invariant under additive shifts, thus, $\|y_i\|_{p,[\tau_i,\tau_{i+1})} = \|y\|_{p,[\tau_i,\tau_{i+1})}$ and $\|R^{y_i}\|_{\p,[\tau_i,\tau_{i+1})} = \|R^{y}\|_{\p,[\tau_i,\tau_{i+1})}$. It therefore holds that
		\begin{equation}\label{eq: ap-1}
			\|Y'\|_{p,[\tau_i,\tau_{i+1})} \leq 4 C_F + \|y'\|_{p,[\tau_i,\tau_{i+1})}
		\end{equation}
		as well as
		\begin{equation}\label{eq: ap-2}
			\|R^Y\|_{\p,[\tau_i,\tau_{i+1})} \leq  1+ \|R^y\|_{\p,[\tau_i,\tau_{i+1})}
		\end{equation}
		for all $i = 0, \ldots, N-1$. Here, $N$ depends on $p$, $\|y,y'\|_{X,p}$, $C_F$, $C_{F,K,X}$, $\|\bX\|_p$ and is, for $p$ fixed, non-decreasing in the other variables. Observe that
		\begin{equation*}
			Y_{t-,t} = y_{t-,t} + \bigg( \int_0^\cdot F_s(Y) \dd \bX_s \bigg)_{t-,t} = y_{t-,t} + F_{t-}(Y) X_{t-,t} + F'_{t-}(Y,Y') \X_{t-,t},
		\end{equation*}
		for any $t \in (0,T]$, so we have
		\begin{equation*}
			R^Y_{t-,t} = R^y_{t-,t} + F'_{t-}(Y,Y') \X_{t-,t}.
		\end{equation*}
		This yields
		\begin{equation*}
			|R^Y_{\tau_{i+1}-,\tau_{i+1}}| \leq \|R^y\|_{p,[\tau_i,\tau_{i+1}]} + (|F'_{\tau_i}(Y,Y')|+ \|F'(Y,Y')\|_{p,[\tau_i,\tau_{i+1})}) |\X_{\tau_{i+1}-,\tau_{i+1}}|.
		\end{equation*}
		Now, we use Assumption~\ref{assumption: existence}~(i), i.e.
		\begin{equation*}
			|F'_{\tau_i}(Y,Y')| + \|F'(Y,Y')\|_{p,[\tau_i,\tau_{i+1})} \leq C_F (1 +\|Y,Y'\|_{X,p,[\tau_i,\tau_{i+1})})^2 (1 + \|X\|_{p,[\tau_i,\tau_{i+1})})^2.
		\end{equation*}
		Since
		\begin{equation*}
			\|Y,Y'\|_{X,p,[\tau_i,\tau_{i+1})} \leq |y'_{\tau_i}| + |F_{\tau_i}(Y)| + \| Y'\|_{p,[\tau_i,\tau_{i+1})} + \|R^Y\|_{\p,[\tau_i,\tau_{i+1})},
		\end{equation*}
		It follows from Assumption~\ref{assumption: existence}~(i), \eqref{eq: ap-1} and~\eqref{eq: ap-2} that
		\begin{equation*}
			\|Y,Y'\|_{X,p,[\tau_i,\tau_{i+1})} \leq |y'_{\tau_i}| + 5 (1 + C_F) + \|y'\|_{p,[\tau_i,\tau_{i+1})} + \| R^y\|_{\p,[\tau_i,\tau_{i+1})}.
		\end{equation*}
		Consequently, there exists a componentwise non-decreasing polynomial $Q_p^{(R)} \colon [0,\infty)^3 \to [0,\infty)$
		such that
		\begin{equation*}
			|R^Y_{\tau_{i+1}-,\tau_{i+1}}| \leq Q_p^{(R)}(\|y,y'\|_{X,p,[\tau_i,\tau_{i+1}]}, C_F, \|\bX\|_{p,[\tau_i,\tau_{i+1}]})
		\end{equation*}
		as well as
		\begin{equation*}
			\|R^Y\|_{\p,[\tau_i,\tau_{i+1}]} \leq 1 + \|R^y\|_{\p,[\tau_i,\tau_{i+1}]} + Q_p^{(R)}(\|y,y'\|_{X,p,[\tau_i,\tau_{i+1}]}, C_F, \|\bX\|_{p,[\tau_i,\tau_{i+1}]})
		\end{equation*}
		for all $i = 0, \ldots, N-1$. Moreover, since
		\begin{equation*}
			Y'_{t-,t} = y'_{t-,t} + F_{t-,t}(Y),
		\end{equation*}
		for any $t \in (0,T]$, we have
		\begin{equation*}
			|Y'_{\tau_{i+1}-,\tau_{i+1}}| \leq \|y'\|_{p,[\tau_i,\tau_{i+1}]} + |F_{\tau_{i+1}-,\tau_{i+1}}(Y)|.
		\end{equation*}
		By Assumption~\ref{assumption: existence}~(i), it holds that
		\begin{equation*}
			|F_{\tau_{i+1}-,\tau_{i+1}}(Y)| \leq C_F (1 + \|Y\|_{p,[\tau_i,\tau_{i+1})} + |Y_{\tau_{i+1}-,\tau_{i+1}}|),
		\end{equation*}
		thus, we need to control the jump of $Y$ at $\tau_{i+1}$. For this, note that
		\begin{align*}
			&|Y_{\tau_{i+1}-,\tau_{i+1}}| \\
			&\quad \leq   |y_{\tau_{i+1}-,\tau_{i+1}}| + |F_{\tau_{i+1}-}(Y)||X_{\tau_{i+1}-,\tau_{i+1}}| + |F'_{\tau_{i+1}-}(Y)| |\X_{\tau_{i+1}-,\tau_{i+1}}| \\
			&\quad \leq  |y_{\tau_{i+1}-,\tau_{i+1}}| +  (|F_{\tau_i}(Y)| + \|F(Y)\|_{p,[\tau_i,\tau_{i+1})}) |X_{\tau_{i+1}-,\tau_{i+1}}| \\
			&\qquad + (|F'_{\tau_i}(Y,Y')| + \|F'(Y,Y')\|_{p,[\tau_i,\tau_{i+1})}) |\X_{\tau_{i+1}-,\tau_{i+1}}| \\
			&\quad \leq \|y,y'\|_{p,[\tau_i,\tau_{i+1}]} + (C_F + C_F(1 + (|Y'_{\tau_i}| + \|Y'\|_{\p,[\tau_i,\tau_{i+1})}) \|X\|_{p,[\tau_i,\tau_{i+1}]} \\
			&\qquad + \|R^Y\|_{\p,[\tau_i,\tau_{i+1})} )\|X\|_{p,[\tau_i,\tau_{i+1}]} \\
			&\qquad + C_F(1 + \|Y,Y'\|_{p,[\tau_i,\tau_{i+1})})^2 (1 + \|X\|_{p,[\tau_i,\tau_{i+1}]})^2 \|\X\|_{\p,[\tau_i,\tau_{i+1}]}.
		\end{align*}
		Using~\eqref{eq: ap-1} and~\eqref{eq: ap-2}, we can now conclude that there exist componentwise non-decreasing polynomials $Q_p^{(Y,J)}, Q_p^{(Y')} \colon [0,\infty)^3 \to [0,\infty)$ such that
		\begin{equation*}
			|F_{\tau_{i+1}-,\tau_{i+1}}(Y)| \leq Q_p^{(Y,J)}(\|y,y'\|_{X,p,[\tau_i,\tau_{i+1}]}, C_F, \|\bX\|_{p,[\tau_i,\tau_{i+1}]})
		\end{equation*}
		as well as
		\begin{equation*}
			\| Y'\|_{p,[\tau_i,\tau_{i+1}]} \leq Q_p^{(Y')}(\|y,y'\|_{X,p,[\tau_i,\tau_{i+1}]}, C_F, \|\bX\|_{p,[\tau_i,\tau_{i+1}]}).
		\end{equation*}
		Combining these estimates, we obtain that
		\begin{equation*}
			|Y_0'| +	\|Y'\|_p	+ \|R^Y\|_{\p} \leq K_p(\|y,y'\|_{X,p}, C_F, \|\bX\|_p),
		\end{equation*}
		which is the assertion.
	\end{proof}
	
	\subsection{Continuity of the It{\^o}--Lyons map}
	
	A fundamental contribution of the theory of rough paths is the insight that the solution map, mapping the input (initial condition, coefficients, driving rough path, \dots) of a rough differential equation to its solution, is continuous with respect to suitable distances on the spaces of controlled paths as well as of rough paths, see e.g.~\cite{Friz2020}. In the context of rough differential equations, this solution map is also known as It{\^o}--Lyons map. In the next theorem we present the local Lipschitz continuity of the It{\^o}--Lyons map for rough functional differential equations, based on the following assumption.
	
	\begin{assumption}\label{assumption: Lipschitz}
		Let $X, \tX \in D^p([0,T];\R^d)$ be given. For $(G,G')\in \{(F,F'),( \tF,\tF')\}$ and $Z\in \{X,\tX\}$ we have: For every $K > 0$, there exist constants $C_G > 0$, which depends on $p$ and the functional $G$, and $C_{G,K,X,\widetilde{X}} > 0$, which additionally depends on $K$, $X$, $\widetilde{X}$ such that the maps
		\begin{equation*}
			(G,G')\colon \cV^p_Z([0,T];\R^k) \to \cV^p_Z([0,T];\cL(\R^d;\R^k))
		\end{equation*}
		satisfy, for all $(Y,Y') \in \cV^p_X$, $(\tY,\tY') \in \cV^p_{\tX}$, and every $0 \leq s < t \leq T$,  \\
		\textup{(i)} the growth conditions:
		\begin{align*}
			&|G_t(Y)| \leq C_G, \\
			&|G_{t-,t}(Y)| \leq C_G (1 + \|Y\|_{p,[s,t)} + |Y_{t-,t}|), \\
			&\|G(Y)\|_{p,[s,t]} \leq C_G (1 + (|Y'_s| + \|Y'\|_{p,[s,t]}) \|Z\|_{p,[s,t]} + \|R^Y\|_{\p,[s,t]}),
			\quad\text{and}\\
			&\|G(Y),G'(Y,Y')\|_{Z,p,[s,t]} \leq C_G (1 + \|Y,Y'\|_{Z,p,[s,t]})^2 (1 + \|Z\|_{p,[s,t]})^2;
		\end{align*}
		\textup{(ii)} the Lipschitz conditions:
		\begin{align*}
			&\|G(Y) - G(\tY)\|_{p,[s,t]} \leq C_{F,K,X,\widetilde{X}} (|Y_s - \tY_s| + \|Y - \tY\|_{p,[s,t]}),
			\quad\text{and}\\
			&\|G(Y),G'(Y,Y');G(\tY),G'(\tY,\tY')\|_{X,\tX,p,[s,t]}\\
			&\quad \leq C_{G,K,X,\widetilde{X}} (|Y_s - \tY_s| + \|Y,Y';\tY,\tY'\|_{X,\tX,p,[s,t]} + \|X - \tX\|_{p,[s,t]}),
		\end{align*}
		if $\|Y,Y'\|_{X,p,[s,t]}, \|\tY,\tY'\|_{\tX,p,[s,t]} \leq K$.
		
		Moreover, there exists a constant $C_{F - \tF}>0$, which depends on $p$ and the functionals $F-\tF$, such that
		\begin{align*}
			&|(F-\tF)_t(Y)| \leq C_{F - \tF}, \\
			&|(F-\tF)_{t-,t}(Y)| \leq C_{F - \tF} (1 + \|Y\|_{p,[s,t)} + |Y_{t-,t}|), \\
			&\|(F-\tF)(Y)\|_{p,[s,t]} \leq C_{F - \tF} (1 + (|Y'_s| + \|Y'\|_{p,[s,t]}) \|X\|_{p,[s,t]} + \|R^Y\|_{\p,[s,t]}),
			\quad\text{and}\\
			&\|(F-\tF)(Y),(F'-\tF')(Y,Y')\|_{X,p,[s,t]} \leq C_{F - \tF} (1 + \|Y,Y'\|_{X,p,[s,t]})^2 (1 + \|X\|_{p,[s,t]})^2.
		\end{align*}
	\end{assumption}
	
	\begin{theorem}\label{thm: continuity}
		Let $\bX, \tbX \in \cD^p([0,T];\R^d)$ be c{\`a}dl{\`a}g $p$-rough paths for $p \in (2,3)$, $(y,y') \in \cV^p_X([0,T];\R^k)$, $(\ty,\ty') \in \cV^p_{\tX}([0,T];\R^k)$ be given controlled paths with respect to $X$ and $\tX$, respectively. Suppose that the non-anticipative functionals $(F,F')$, $(\tF,\tF')$ satisfy Assumption~\ref{assumption: Lipschitz} given $X,\tX$.
		
		Let $(Y,Y) \in \cV^p_X([0,T];\R^k)$ be the solution given by Theorem~\ref{thm: existence and uniqueness} to the rough functional differential equation~\eqref{eq: RFDE}, and $(\tY,\tY') \in \cV^p_{\tX}([0,T];\R^k)$ be the solution to the rough functional differential equation~\eqref{eq: RFDE} driven by $\tbX$ with initial condition $(\ty,\ty')$ and functional $(\tF,\tF')$, and suppose that $\|Y,Y'\|_{X,p}, \|\tY,\tY'\|_{\tX,p} \leq K$, for some $K > 0$. Then, we have the estimate
		\begin{align*}
			&|Y_0 - \tY_0| + \|Y,Y';\tY,\tY'\|_{X,\tX,p} \\
			&\quad \lesssim |y_0 - \ty_0| + |F_0(y) - \tF_0(\ty)| + \|y,y';\ty,\ty'\|_{X,\tX,p} + C_{F-\tF} + \|\bX;\tbX\|_p,
		\end{align*}
		where the implicit multiplicative constant depends on $p$, $C_F \vee C_{\tF}$, $\C \vee C_{\tF,K,X,\tX}$, $K$, $\|\bX\|_p$ and $\|\tbX\|_p$.
	\end{theorem}
	
	\begin{proof}
		\emph{Step~1: Local estimate.} Let $(Y,Y') \in \cV^p_X([0,T];\R^k)$, $(\tY,\tY') \in \cV^p_{\tX}([0,T];\R^k)$ be the global solutions to the RFDE~\eqref{eq: RFDE}, with data $((y,y'), (F,F'), \bX)$, $((\ty,\ty'), (\tF,\tF'), \tbX)$, respectively, see Theorem~\ref{thm: existence and uniqueness}. Let $t \in (0,T]$. Without loss of generality assume that $C_{\tF} \leq C_F$, $C_{\tF,K,X,\tX} \leq \C$. As
		\begin{equation*}
			\|Y- \widetilde{Y}\|_{p,[0,t]} \leq (|Y'_0-\widetilde{Y}'_0|+\|Y'-\widetilde{Y}'\|_{p,[0,t]}) \|X\|_{p,[0,t]}+(|\widetilde{Y}'_0|+\|\widetilde{Y}'\|_{p,[0,t]}) \|X-\widetilde{X}\|_{p,[0,t]},
		\end{equation*}
		Assumption~\ref{assumption: Lipschitz} gives that
		\begin{align*}
			&\|F(Y) - \tF(\tY)\|_{p,[0,t]} \\
			&\quad \leq \|F(Y) - F(\tY)\|_{p,[0,t]} + \|(F-\tF)(\tY)\|_{p,[0,t]} \\
			&\quad \leq \C (|Y_0 - \tY_0| + \|Y - \tY\|_{p,[0,t]}) \\
			&\qquad + C_{F-\tF} (1 + (|\tY'_0| + \|\tY'\|_{p,[0,t]}) \|X\|_{p,[0,t]} + \|R^{\tY}\|_{\p,[0,t]}) \\
			&\quad \lesssim_p \C (|Y_0 - \tY_0| + (|Y'_0 - \tY'_0| + \|Y' - \tY'\|_{p,[0,t]}) (\|X\|_{p,[0,t]} \vee \|\tX\|_{p,[0,t]})) \\
			&\qquad + \C \|R^y - R^{\ty}\|_{\p,[0,t]} + \C \|R^{\int_0^\cdot F(Y) \d \bX - \int_0^\cdot \tF(\tY) \d \tbX}\|_{\p,[0,t]} \\
			&\qquad + \C \|X-\tX\|_{p,[0,t]} + C_{F-\tF} (1 + K) (1 + \|X\|_{p,[0,t]} \vee \|\tX\|_{p,[0,t]}).
		\end{align*}
		Further, by Lemma~\ref{lemma: estimate rough integral} and Lemma~\ref{lemma: Lipschitz estimate rough integral} we have that
		\begin{align*}
			&\|R^{\int_0^\cdot F(Y) \d \bX - \int_0^\cdot \tF(\tY) \d \tbX}\|_{\p,[0,t]} \\
			&\quad \leq \|R^{\int_0^\cdot F(Y) \d \bX - \int_0^\cdot F(\tY) \d \tbX}\|_{\p,[0,t]} + \|R^{\int_0^\cdot (F-\tF)(\tY) \d \tbX}\|_{\p,[0,t]} \\
			&\quad \lesssim \C (|Y_0 - \tY_0| + \|Y,Y';\tY,\tY'\|_{X,\tX,p,[0,t]} + \|X - \tX\|_{p,[0,t]}) (\|\bX\|_{p,[0,t]} \vee \|\tbX\|_{p,[0,t]}) \\
			&\qquad + C_F (1 + K)^2 (1 + \|X\|_{p,[0,t]} \vee \|\tX\|_{p,[0,t]})^2 \|\bX;\tbX\|_{p,[0,t]} \\
			&\qquad + C_{F-\tF} (1 + K)^2 (1 + \|X\|_{p,[0,t]} \vee \|\tX\|_{p,[0,t]})^2 (\|\bX\|_{p,[0,t]} \vee \|\tbX\|_{p,[0,t]}),
		\end{align*}
		where the implicit multiplicative constant depends on $p$, $\|\bX\|_p$ and $\|\tbX\|_p$. Combining the results, we get that
		\begin{align*}
			&\|Y' - \tY'\|_{p,[0,t]} + \|R^Y - R^{\tY}\|_{\p,[0,t]} \\
			&\quad \leq \|y' - \ty'\|_{p,[0,t]} + \|F(Y) - \tF(\tY)\|_{p,[0,t]} \\
			&\qquad + \|R^y - R^{\ty}\|_{\p,[0,t]} + \|R^{\int_0^\cdot F(Y) \d \bX - \int_0^\cdot \tF(\tY) \d \tbX}\|_{\p,[0,t]} \\
			&\quad \leq C_1 (\|\bX\|_{p,[0,t]} \vee \|\tbX\|_{p,[0,t]}) (\|Y' - \tY'\|_{p,[0,t]} + \|R^Y - R^{\tY}\|_{\p,[0,t]}) \\
			&\qquad + C_2 (|Y_0 - \tY_0| + |F_0(Y) - \tF_0(\tY)| + \|y,y';\ty,\ty'\|_{X,\tX,p,[0,t]} + C_{F-\tF} + \|\bX;\tbX\|_{p,[0,t]})
		\end{align*}
		for some constants $C_1>0$, which depends on $p$, $\C$, $\|\bX\|_p$ and $\|\tbX\|_p$, and $C_2 > 1$, which depends additionally on $C_F$ and $K$. Hence, we can choose $t$ sufficiently small such that $C_1 (\|\bX\|_{p,[0,t]} \vee \|\tbX\|_{p,[0,t]}) \leq \frac{1}{2}$, which implies that
		\begin{equation}\label{eq: local continuity estimate}
			\begin{split}
				&\|Y' - \tY'\|_{p,[0,t]} + \|R^Y - R^{\tY}\|_{\p,[0,t]} \\
				&\quad \lesssim |Y_0 - \tY_0| + |F_0(Y) - \tF_0(\tY)| + \|y,y';\ty,\ty'\|_{X,\tX,p,[0,t]} + C_{F-\tF} + \|\bX;\tbX\|_{p,[0,t]}.
			\end{split}
		\end{equation}
		\emph{Step~2: Global estimate.} Recall the right-continuous control function $w \colon \Delta_T \to [0,\infty)$ given by
		\begin{equation*}
			w(s,t) := \|X\|_{p,[s,t]}^p + \|\X\|_{\p,[s,t]}^{\p}, \quad (s,t) \in \Delta_T,
		\end{equation*}
		as introduced in the proof of Theorem~\ref{thm: existence and uniqueness}, and let $\tilde{w}(s,t) := \|\tX\|_{p,[s,t]}^p + \|\tbbX\|_{\p,[s,t]}^{\p}$, $(s,t) \in \Delta_T$. We infer from Step~1 that there exists a constant $\gamma > 0$, which depends on $p$, $\C$, $\|\bX\|_p$ and $\|\tbX\|_p$, such that on any interval $[s,t]$ with $w(s,t) \vee \tilde{w}(s,t) \leq \gamma$ the local solutions satisfy an estimate of the form~\eqref{eq: local continuity estimate}.
		
		Let $c(s,t) := w(s,t) + \tilde{w}(s,t)$, $(s,t) \in \Delta_T$. Since $c$ is right-continuous, there exists a partition $\mathcal{P} = \{0 = t_0 < \dots < t_N = T\}$ of $[0,T]$, such that
		\begin{equation*}
			c(t_i,t_{i+1}-) = \gamma, \quad \text{or} \quad c(t_i,t_{i+1}-) < \gamma \quad \text{and} \quad c(t_i,t_{i+1}-) + c(t_{i+1}-,t_{i+1}) \geq \gamma,
		\end{equation*}
		for every $i = 0, 1, \ldots, N-1$. Since $w$ and $\tilde{w}$ and, thus, $c$ is superadditive, we have that
		\begin{equation*}
			N \gamma \leq \sum_{i=0}^{N-1} c(t_i, t_{i+1}-) + c(t_{i+1}-,t_{i+1}) \leq c(0,T).
		\end{equation*}
		Therefore, the number of partition points $N$ may be bounded by a constant depending only on $\gamma$, $w(0,T)$ and $\tilde{w}(0,T)$. Thus, in this step, we may combine the local estimates on each of the subintervals $[t_i,t_{i+1})$, together with simple estimates on the jumps at the end-points of these subintervals, which we aim to derive, to obtain the global estimate. More precisely, by Step~1, we have the local estimate
		\begin{equation}\label{eq: ito-lip-pre}
			\begin{split}
				&\|Y' - \tY'\|_{p,[t_i,t_{i+1})} + \|R^Y - R^{\tY}\|_{\p,[t_i,t_{i+1})} \\
				&\quad \lesssim |Y_{t_i} - \tY_{t_i}| + |F_{t_i}(Y) - \tF_{t_i}(\tY)| + |y'_{t_i} - \ty'_{t_i}| \\
				&\qquad + \|y' - \ty'\|_{p,[t_i,t_{i+1})} + \|R^y - R^{\ty}\|_{\p,[t_i,t_{i+1})} + C_{F-\tF} + \|\bX;\tbX\|_{p,[t_i,t_{i+1})},
			\end{split}
		\end{equation}
		for $i=0, \ldots, N-1$, where the implicit multiplicative constant depends on $p$ and $\C$, $C_F$, $K$, $\|\bX\|_p$, $\|\tbX\|_p$, but not on the index $i$. So, it remains to bound
		\begin{equation*}
			|Y'_{t_{i+1}-,t_{i+1}} - \tY'_{t_{i+1}-,t_{i+1}}| + |R^Y_{t_{i+1}-,t_{i+1}} - R^{\tY}_{t_{i+1}-,t_{i+1}}|
		\end{equation*}
		to extend the previous estimate to $[t_i,t_{i+1}]$.
		
		We note that $(\int_0^\cdot F_s(Y) \dd \bX_s)_{t-,t} = F_{t-}(Y) X_{t-,t} + F'_{t-}(Y,Y') \X_{t-,t}$, that is, with $Y_{t-,t} = y_{t-,t} + (\int_0^\cdot F_s(Y) \dd \bX_s)_{t-,t}$, it follows that $Y'_{t-,t} = y'_{t-,t} + F_{t-,t}(Y)$ and $R^Y_{t-,t} = R^y_{t-,t} + F'_{t-}(Y,Y') \X_{t-,t}$, for $t \in (0,T]$. Given the assumptions, we then have
		\begin{equation}\label{eq: local continuity estimate 2}
			\begin{split}
				&\|Y' - \tY'\|_{p,[t_i,t_{i+1}]} + \|R^Y - R^{\tY}\|_{\p,[t_i,t_{i+1}]} \\
				&\quad \leq \|Y' - \tY'\|_{p,[t_i,t_{i+1})} + \|R^Y - R^{\tY}\|_{\p,[t_i,t_{i+1})}  \\
				&\qquad + |Y'_{t_{i+1}-,t_{i+1}} - \tY'_{t_{i+1}-,t_{i+1}}| + |R^Y_{t_{i+1}-,t_{i+1}} - R^{\tY}_{t_{i+1}-,t_{i+1}}| \\
				&\quad \lesssim |Y_{t_i} - \tY_{t_i}| + |F_{t_i}(Y) - \tF_{t_i}(\tY)| + \|y,y';\ty,\ty'\|_{X,\tX,p,[t_i,t_{i+1}]} + C_{F-\tF} + \|\bX;\tbX\|_{p,[t_i,t_{i+1}]},
			\end{split}
		\end{equation}
		where the implicit multiplicative constant depends on $p$, $C_F$, $\C$, $K$, $\|\bX\|_p$ and $\|\tbX\|_p$, and not on the index $i$. Here, we used Assumption~\ref{assumption: Lipschitz} and the estimate~\eqref{eq: ito-lip-pre} to derive that
		\begin{align*}
			&|(R^Y_{t_{i+1}-,t_{i+1}} - R^{\tY}_{t_{i+1}-,t_{i+1}}) - (R^y_{t_{i+1}-,t_{i+1}} - R^{\ty}_{t_{i+1}-,t_{i+1}})| \\
			&\quad \leq |F'_{t_{i+1}-}(Y,Y') \X_{t_{i+1}-,t_{i+1}} - \tF'_{t_{i+1}-}(\tY,\tY') \tbbX_{t_{i+1}-,t_{i+1}}| \\
			&\quad \leq |F'_{t_{i+1}-}(Y,Y') - \tF'_{t_{i+1}-}(\tY,\tY')| |\X_{t_{i+1}-,t_{i+1}}| + |\tF_{t_{i+1}-}'(\tY,\tY')| |\X_{t_{i+1-},t_{i+1}} - \tbbX_{t_{i+1}-,t_{i+1}}| \\
			&\quad \leq (|F'_{t_i}(Y,Y') - \tF'_{t_i}(\tY,\tY)| + \|F'(Y,Y') - \tF'(\tY,\tY')\|_{p,[t_i,t_{i+1})})\|\X\|_{\p,[t_i,t_{i+1}]} \\
			&\qquad + (|\tF'_{t_i}(\tY,\tY')| + \|\tF'(\tY,\tY')\|_{p,[t_i,t_{i+1})}) \|\X - \tbbX\|_{\p,[t_i,t_{i+1}]} \\
			&\quad \lesssim |Y_{t_i} - \tY_{t_i}| + |F_{t_i}(Y) - \tF_{t_i}(\tY)| + \|y,y';\ty,\ty'\|_{X,p,[t_i,t_{i+1})} + C_{F-\tF} + \|\bX;\tbX\|_{p,[t_i,t_{i+1}]},
		\end{align*}
		and
		\begin{align*}
			&|(Y'_{t_{i+1}-,t_{i+1}} - \tY'_{t_{i+1}-,t_{i+1}}) - (y'_{t_{i+1}-,t_{i+1}} - \ty'_{t_{i+1}-,t_{i+1}})| \\
			&\quad = |F_{t_{i+1}-,t_{i+1}}(Y) - \tF_{t_{i+1}-,t_{i+1}}(\tY) | \\
			&\quad \leq |(F-\tF)_{t_{i+1}-,t_{i+1}}(Y)| + |\tF_{t_{i+1}-,t_{i+1}}(Y) - \tF_{t_{i+1}-,t_{i+1}}(\tY)| \\
			&\quad \leq C_{F-\tF} (1 + (|Y'_{t_i}| + \|Y'\|_{p,[t_i,t_{i+1}]}) \|X\|_{p,[t_i,t_{i+1}]} + \|R^Y\|_{\p,[t_i,t_{i+1}]}) \\
			&\qquad + \C (|Y_{t_i} - \tY_{t_i}| + \|Y - \tY\|_{p,[t_i,t_{i+1})}+ |Y_{t_{i+1}-,t_{i+1}} - \tY_{t_{i+1}-,t_{i+1}}|) \\
			&\quad \lesssim |Y_{t_i} - \tY_{t_i}| + |F_{t_i}(Y) - \tF_{t_i}(\tY)| + \|y,y';\ty,\ty'\|_{X,p,[t_i,t_{i+1})} + C_{F-\tF} + \|\bX;\tbX\|_{p,[t_i,t_{i+1}]},
		\end{align*}
		where the implicit multiplicative constant depends on $p$, $C_F$, $\C$, $K$ and $\|\bX\|_p \allowbreak \vee \|\tbX\|_p$, and not on the index $i$.
		
		Now, we need to control the term  $|Y_{t_i} - \tY_{t_i}| + |F_{t_i}(Y) - \tF_{t_i}(\tY)| + |y'_{t_i} - \ty'_{t_i}|$. For this, we note that
		\begin{align*}
			&|Y_{t_i} - \tY_{t_i}| + |F_{t_i}(Y) - \tF_{t_i}(\tY)| + |y'_{t_i} - \ty'_{t_i}| \\
			&\quad \leq |Y_{t_{i-1}} - \tY_{t_{i-1}}| + |F_{t_{i-1}}(Y) - \tF_{t_{i-1}}(\tY)| + |y'_{t_{i-1}} - \ty'_{t_{i-1}}| \\
			&\qquad + \|Y - \tY\|_{p,[t_{i-1},t_i]} + \|F(Y) - \tF(\tY)\|_{p,[t_{i-1},t_i]} + \|y' - \ty'\|_{p,[t_{i-1},t_i]} \\
			&\quad \lesssim |Y_{t_{i-1}} - \tY_{t_{i-1}}| + |F_{t_{i-1}}(Y) - \tF_{t_{i-1}}(\tY)| + \|y,y';\ty,\ty'\|_{X,\tX,p,[t_{i-1},t_i]} \\
			&\qquad + C_{F-\tF} + \|\bX;\tbX\|_{p,[t_{i-1},t_i]} \\
			&\qquad + \|Y' - \tY'\|_{p,[t_{i-1},t_i]} + \|R^Y - R^{\tY}\|_{\p,[t_{i-1},t_i]},
		\end{align*}
		where the implicit multiplicative constant depends on $p$, $C_F$, $\C$, $K$, $\|X\|_p$ and $\|\tX\|_p$, and not on the index $i$; thus applying the estimate~\eqref{eq: local continuity estimate 2} for $i-1$ gives that
		\begin{align*}
			&|Y_{t_i} - \tY_{t_i}| + |F_{t_i}(Y) - \tF_{t_i}(\tY)| + |y'_{t_i} - \ty'_{t_i}| \\
			&\quad \lesssim |Y_{t_{i-1}} - \tY_{t_{i-1}}| + |F_{t_{i-1}}(Y) - \tF_{t_{i-1}}(\tY)| + \|y,y';\ty,\ty'\|_{X,\tX,p,[t_{i-1},t_i]} \\
			&\qquad + C_{F-\tF} + \|\bX;\tbX\|_{p,[t_{i-1},t_i]}.
		\end{align*}
		Iteratively, we obtain for any $i = 1, \ldots, N$ that
		\begin{align*}
			&|Y_{t_i} - \tY_{t_i}| + |F_{t_i}(Y) - \tF_{t_i}(\tY)| + |y'_{t_i} - \ty'_{t_i}| \\
			&\quad \lesssim |y_0 - \ty_0| + |F_0(y) - \tF_0(\ty)| + |y'_0 - \ty'_0| \\
			&\qquad + \sum_{j=0}^{i-1} \Big( \|y' - \ty'\|_{p,[t_j,t_{j+1}]} +\|R^y - R^{\ty}\|_{\p,[t_j,t_{j+1}]} + C_{F-\tF} + \|\bX;\tbX\|_{p,[t_j,t_{j+1}]} \Big),
		\end{align*}
		that is,
		\begin{align*}
			&(|Y_{t_i} - \tY_{t_i}| + |F_{t_i}(Y) - \tF_{t_i}(\tY)| + |y'_{t_i} - \ty'_{t_i}|)^p \\
			&\quad \lesssim (|y_0 - \ty_0| + |F_0(y) - \tF_0(\ty)| + |y'_0 - \ty'_0| + N C_{F-\tF})^p \\
			&\qquad + \sum_{j=0}^{i-1} \Big( \|y' - \ty'\|_{p,[t_j,t_{j+1}]}^p +\|R^y - R^{\ty}\|_{\p,[t_j,t_{j+1}]}^p + \|\bX;\tbX\|_{p,[t_j,t_{j+1}]}^p \Big).
		\end{align*}
		This implies that
		\begin{align*}
			&|Y_{t_i} - \tY_{t_i}| + |F_{t_i}(Y) - \tF_{t_i}(\tY)| + |y'_{t_i} - \ty'_{t_i}|\\
			&\quad \lesssim |y_0 - \ty_0| + |F_0(y) - \tF_0(\ty)| + \|y,y';\ty,\ty'\|_{X,\tX,p} + C_{F-\tF} + \|\bX;\tbX\|_p,
		\end{align*}
		which is the desired control.
		
		If we plug this into~\eqref{eq: local continuity estimate 2}, it follows that
		\begin{align*}
			&\|Y' - \tY'\|_{p,[t_i,t_{i+1}]} + \|R^Y - R^{\tY}\|_{\p,[t_i,t_{i+1}]} \\
			&\quad \lesssim |y_0 - \ty_0| + |F_0(y) - \tF_0(\ty)| + \|y,y';\ty,\ty'\|_{X,\tX,p} + C_{F-\tF} + \|\bX;\tbX\|_p.
		\end{align*}
		Since $\|\cdot\|_{p,[0,T]} \leq N \sum_{i=0}^{N-1} \|\cdot\|_{p,[t_i,t_{i+1}]}$ for any $p \geq 1$, see e.g.~\cite[Lemma~A.1]{Allan2021b}, the estimate finally follows.
	\end{proof}

	\section{Examples of RFDEs}\label{sec: examples}
	
	The general framework of rough functional differential equations, presented in Section~\ref{sec: RFDE}, allows to treat various classes of rough differential equations. In this section, some exemplary rough functional differential equations are discussed, aiming to develop the main conceptional ideas and demonstrating the scope of RFDEs rather than pushing for the most general results.
	
	\subsection{Classical RDEs}
	
	Let us start with the classical rough differential equation (RDE)
	\begin{equation}\label{eq: classical RDE}
		Y_t = y_0 + \int_0^t f(Y_s) \dd \bX_s, \quad t \in [0,T],
	\end{equation}
	where $y_0 \in \R^k$, $f \in C^3_b(\R^k;\cL(\R^d;\R^k))$ and $\bX \in \cD^p([0,T];\R^d)$ for $p \in (2,3)$. While the existence and uniqueness of solutions to the RDE~\eqref{eq: classical RDE} driven by a continuous rough path and the continuity of the solution map were first proven by Lyons~\cite{Lyons1998}, the analogous results for RDEs driven by c{\`a}dl{\`a}g rough paths were more recently obtained by Friz and Zhang~\cite{Friz2018}. As an application of Theorem~\ref{thm: existence and uniqueness} and Theorem~\ref{thm: continuity}, one can recover these results, demonstrating that Assumption~\ref{assumption: existence} and~\ref{assumption: Lipschitz} are, indeed, natural generalizations of the classical assumptions of the coefficients of a rough differential equation. Furthermore, note that Corollary~\ref{cor: classical RDE} presents the continuity of the solution map with respect to the controlled path norm, which slightly generalizes \cite[Theorem~3.8]{Friz2018}.
	
	\begin{corollary}\label{cor: classical RDE}~
		\begin{itemize}
			\item[(i)]
			If $f \in C^3_b(\R^k;\cL(\R^d;\R^k))$, there exists a unique solution to the RDE~\eqref{eq: classical RDE}. Moreover, there exists a non-decreasing function $K_p \colon [0, \infty)^2 \to [0,\infty)$ such that
			\begin{equation*}
				\|Y,Y'\|_{X,p} \leq K_p(\|f\|_{C^2_b}, \|\bX\|_p).
			\end{equation*}
			\item[(ii)]
			Let $(Y,Y') \in \cV^p_X([0,T];\R^k)$ be the unique solution to the RDE~\eqref{eq: classical RDE}. Moreover, let $\ty_0 \in \R^k$, $\tf \in C^3_b(\R^k;\cL(\R^d;\R^k))$, $\tbX \in \cD^p([0,T];\R^d)$ with corresponding solution $(\tY,\tY') \in \cV^p_{\tX}([0,T];\R^d)$, and suppose that $\|Y,Y'\|_{X,p}, \|\tY,\tY'\|_{\tX,p} \leq K$, for some ${K > 0}$. Then, we have the estimate
			\begin{equation*}
				|Y_0 - \tY_0| + \|Y,Y';\tY,\tY'\|_{X,\tX,p} \lesssim |y_0 - \ty_0| + \|f - \tf\|_{C^2_b} + \|\bX;\tbX\|_p,
			\end{equation*}
			where the implicit multiplicative constant depends on $p$, $\|f\|_{C^3_b} \vee \|\tf\|_{C^3_b}$, $K$, $\|\bX\|_p$ and $\|\tbX\|_p$.
		\end{itemize}
	\end{corollary}
	
	In order to apply the existence and uniqueness result presented in Theorem~\ref{thm: existence and uniqueness}, and also the continuity result presented in Theorem~\ref{thm: continuity}, we need to check that the vector field $f$ in the RDE~\eqref{eq: classical RDE} satisfies Assumption~\ref{assumption: existence} and~\ref{assumption: Lipschitz}. This is the content of the next lemma, which we formulate slightly more general, with regard to the dimensions of the underlying spaces, for later use.
	
	\begin{lemma}\label{lemma: transformation of controlled path}
		Let $f \in C^3_b(\R^m;\cL(\R^d;\R^k))$ and $X, \tX \in D^p([0,T];\R^d)$. The non-anticipative functional
		\begin{equation*}
			(F,F') \colon \cV^p_X([0,T];\R^m) \to \cV^p_X([0,T];\cL(\R^d;\R^k)), \quad (F(Y),F'(Y,Y')) := (f(Y), \D f(Y) Y'),
		\end{equation*}
		satisfies Assumption~\ref{assumption: Lipschitz}~(i) and (ii), and, in particular, Assumption~\ref{assumption: existence}, given $X,\tX$.
	\end{lemma}
	
	\begin{proof}
		Since the proof is fairly standard, we provide only a sketch of a proof, following, e.g., the proofs of~\cite[Lemma~3.5, Lemma~3.6, Lemma~3.7]{Friz2018}.
		
		Fix $(s,t) \in \Delta_T$ and let $(Y,Y') \in \cV^p_X$, $(\tY,\tY') \in \cV^p_{\tX}$.
		
		\emph{Growth conditions.} It is clear that $|F_t(Y)| \leq \|f\|_{C^2_b}$, and it follows from the Lipschitz continuity of $f$ that
		\begin{equation*}
			|F_{t-,t}(Y)| \leq \|F(Y)\|_{p,[s,t]} \leq \|f\|_{C^2_b} \|Y\|_{p,[s,t]}.
		\end{equation*}
		We now note that $\|Y\|_{p,[s,t]} \leq \|Y\|_{p,[s,t)} + |Y_{t-,t}|$ as well as $\|Y\|_{p,[s,t]} \leq C_p(1 + (|Y'_s| + \|Y'\|_{p,[s,t]}) \cdot \|X\|_{p,[s,t]} + \|R^Y\|_{\p,[s,t]})$.
		Further, it holds that
		\begin{align*}
			&|F'_s(Y,Y')| + \|F'(Y,Y')\|_{p,[s,t]} = |\D f(Y_s) Y'_s| + \|\D f(Y) Y'\|_{p,[s,t]} \\
			&\quad \leq \|f\|_{C^2_b} (|Y'_s| + \|Y'\|_{p,[s,t]}) (1 + \|Y\|_{p,[s,t]}) \\
			&\quad \lesssim_p \|f\|_{C^2_b} (|Y'_s| + \|Y'\|_{p,[s,t]}) (1 + (|Y'_s| + \|Y'\|_{p,[s,t]}) \|X\|_{p,[s,t]} + \|R^Y\|_{\p,[s,t]}) \\
			&\quad \lesssim_p \|f\|_{C^2_b} (1 + \|Y,Y'\|_{X,p,[s,t]}) (1 + \|X\|_{p,[s,t]})
		\end{align*}
		and by Taylor's expansion,
		\begin{align*}
			R^{F(Y)}_{s,t} &= R^{f(Y)}_{s,t} = f(Y_t) - f(Y_s) - \D f(Y_s)Y_{s,t} + \D f(Y_s) R^Y_{s,t} \\
			&= \frac{1}{2} \D^2 f(Y_s + \lambda Y_{s,t}) Y^2_{s,t} + \D f(Y_s) R^Y_{s,t},
		\end{align*}
		with $\lambda \in [0,1]$, which implies that
		\begin{align*}
			&\|R^{F(Y)}\|_{\p,[s,t]} = \|R^{f(Y)}\|_{\p,[s,t]} \leq \|f\|_{C^2_b} (\|Y\|^2_{p,[s,t]} + \|R^Y\|_{\p,[s,t]}) \\
			&\quad \lesssim_p \|f\|_{C^2_b} (((|Y'_s| + \|Y'\|_{p,[s,t]}) \|X\|_{p,[s,t]} + \|R^Y\|_{\p,[s,t]})^2 + \|R^Y\|_{\p,[s,t]}) \\
			&\quad \lesssim_p \|f\|_{C^2_b} (1 + \|Y,Y'\|_{X,p,[s,t]})^2 (1 + \|X\|_{p,[s,t]})^2.
		\end{align*}
		Assumption~\ref{assumption: Lipschitz}~(i) therefore holds with some constant $C_F = \|f\|_{C^2_b}$ up to a multiplicative constant which depends on $p$.
		
		\emph{Lipschitz conditions.} Fix $K > 0$ and assume that $\|Y,Y'\|_{X,p,[s,t]}, \|\tY,\tY'\|_{\tX,p,[s,t]} \leq K$. The proofs work verbatim as the proofs of~\cite[Lemma~3.1 and Lemma~3.7]{Friz2018}. The constant $\C$ depends on $p$ and $\|f\|_{C^3_b}$, $K$, $\|X\|_p$ and $\|\tX\|_p$.
	\end{proof}
	
	\begin{proof}[Proof of Corollary~\ref{cor: classical RDE}]
		\emph{(i)} The existence and uniqueness of the solution follows immediately from Lemma~\ref{lemma: transformation of controlled path} and Theorem~\ref{thm: existence and uniqueness}. For the a priori estimate, note that $\|y,y'\|_{X,p} = 0$ and $C_F \lesssim_p \|f\|_{C^2_b}$.
		
		\emph{(ii)} To apply the continuity result presented in Theorem~\ref{thm: continuity}, we need to ensure that the functionals satisfy the required assumptions. For $(F,F'), (\tF,\tF')$, this is given in Lemma~\ref{lemma: transformation of controlled path}, and further,
		\begin{align*}
			&(F-\tF,(F-\tF)') \colon \cV^p_X([0,T];\R^k) \to \cV^p_X([0,T];\cL(\R^d;\R^k)), \\
			&((F-\tF)(Y), (F-\tF)'(Y,Y')) := (f(Y) - \tf(Y), \D f(Y) Y' - \D \tf(Y) Y'),
		\end{align*}
		satisfies the corresponding estimates in Assumption~\ref{assumption: Lipschitz}, since $C^3_b$ is a vector space. Thus we have, $C_{F-\tF} \lesssim_p \|f - \tf\|_{C^2_b}$.
	\end{proof}
	
	\subsection{Controlled RDEs}
	
	Motivated by pathwise stochastic control, see e.g.~\cite{Diehl2017, Allan2020}, and robust stochastic filtering, see e.g.~\cite{Allan2021b}, as well as analogously to controlled stochastic differential equations, see e.g.~\cite{Pham2009}, we consider the controlled rough differential equation
	\begin{equation}\label{eq: controlled RDE}
		Y_t = y_t + \int_0^t f(\alpha_s, Y_s) \dd \bX_s, \quad t \in [0,T],
	\end{equation}
	where $\bX \in \cD^p([0,T];\R^d)$ for $p \in (2,3)$, $(y,y') \in \cV^p_X([0,T];\R^k)$, $f \in C^3_b(\R^{k+e};\cL(\R^d;\R^k))$, and $(\alpha,\alpha') \in \cV^p_X([0,T];\R^e)$ is a fixed controlled path, with $e \in \N$. In case of continuous rough paths and controls $\alpha$ of finite $\p$-variation, controlled RDEs were treated in~\cite[Theorem~2.3]{Allan2021b}. The following corollary provides an existence, uniqueness and continuity result for controlled RDEs driven by c{\`a}dl{\`a}g $p$-rough paths and with controls $\alpha$, which are only required to be controlled paths.
	
	\begin{corollary}\label{cor: controlled RDE}~
		\begin{itemize}
			\item[(i)] If $f \in C^3_b(\R^{k+e};\cL(\R^d;\R^k))$ and $(\alpha,\alpha') \in \cV^p_X([0,T];\R^e)$, then there exists a unique solution to the controlled rough differential equation~\eqref{eq: controlled RDE}. Moreover, there exists a componentwise non-decreasing function $K_p \colon [0, \infty)^5 \to [0,\infty)$ such that
			\begin{equation*}
				\|Y,Y'\|_{X,p} \leq K_p(\|f\|_{C^2_b}, \|y,y'\|_{X,p}, \|\alpha\|_p, \|\alpha,\alpha'\|_{X,p}, \|\bX\|_p).
			\end{equation*}
			\item[(ii)] Let $(Y,Y') \in \cV^p_X([0,T];\R^k)$ be the unique solution to the controlled rough differential equation~\eqref{eq: controlled RDE}. Moreover, let $(\ty,\ty') \in \cV^p_X([0,T];\R^k)$, $(\tilde{\alpha}, \tilde{\alpha}') \in \cV^p_X([0,T];\R^e)$, $\tf \in C^3_b(\R^{k+e};\cL(\R^d;\R^k))$, with corresponding solution $(\tY,\tY') \in \cV^p_X([0,T];\R^k)$, and suppose that $\|Y,Y'\|_{X,p}, \|\tY,\tY'\|_{X,p} \leq K$, for some $K > 0$. Then, we have the estimate
			\begin{align*}
				&|Y_0 - \tY_0| + \|Y,Y';\tY,\tY'\|_{X,p} \\
				&\quad \lesssim |y_0 - \ty_0| + \|y,y';\ty,\ty'\|_{X,p} + \|f - \tf\|_{C^2_b} + |\alpha_0 - \tilde{\alpha}_0|  + \|\alpha,\alpha';\tilde{\alpha},\tilde{\alpha}\|_{X,p},
			\end{align*}
			where the implicit multiplicative constant depends on $p$, $\|f\|_{C^3_b} \vee \|\tf\|_{C^3_b}$,  $\|\alpha,\alpha'\|_{X,p}$, $\|\tilde{\alpha},\tilde{\alpha}'\|_{X,p}$, $K$, and $\|\bX\|_p$.
		\end{itemize}
	\end{corollary}
	
	In order to apply the existence and uniqueness result presented in Theorem~\ref{thm: existence and uniqueness}, and also the continuity result in Theorem~\ref{thm: continuity}, we need to check that the vector field in the RDE~\eqref{eq: controlled RDE} satisfies Assumption~\ref{assumption: Lipschitz}. This is the content of the next lemma. Note that it will be sufficient to check Assumption~\ref{assumption: Lipschitz}~(ii) for $X = \tX$, since we do not establish stability results with respect to the driving rough path in this subsection. (This also applies to the following subsections.)
	
	\begin{lemma}\label{lemma: transformation of controlled path with control}
		Let $f \in C^3_b(\R^{k+e};\cL(\R^d;\R^k))$ and $X \in D^p([0,T];\R^d)$ for $p \in (2,3)$. Further, let $(\alpha,\alpha') \in \cV^p_X([0,T];\R^e)$. The non-anticipative functional
		\begin{align*}
			&(F,F')\colon \cV^p_X([0,T];\R^k) \to \cV^p_X([0,T];\cL(\R^d;\R^k)), \\
			&(F(Y),F'(Y,Y')) := (f((\alpha,Y)), \D f((\alpha,Y)) (\alpha',Y')),
		\end{align*}
		satisfies Assumption~\ref{assumption: Lipschitz} given $X = \tX$.
	\end{lemma}
	
	\begin{proof}
		Fix $(s,t) \in \Delta_T$ and let $(Y,Y') \in \cV^p_X$. It is clear that $|F_t(Y)| \leq \|f\|_{C^2_b}$, and we note that
		\begin{align*}
			&|F_{t-,t}(Y)| \leq \|F(Y)\|_{p,[s,t]} = \|f((\alpha,Y))\|_{p,[s,t]} \\
			&\quad \leq \|f\|_{C^2_b} \|(\alpha,Y)\|_{p,[s,t]} \\
			&\quad \leq \|f\|_{C^2_b} (1 + \|\alpha\|_{p,[s,t]}) (1 + \|Y\||_{p,[s,t]}),
		\end{align*}
		and it holds that $\|Y\|_{p,[s,t]} \leq \|Y\|_{p,[s,t)} + |Y_{t-,t}|$ as well as $\|Y\|_{p,[s,t]} \leq C_p(1 + (|Y'_s| + \|Y'\|_{p,[s,t]}) \cdot \|X\|_{p,[s,t]} + \|R^Y\|_{\p,[s,t]})$. Applying Lemma~\ref{lemma: transformation of controlled path} to the enlarged controlled path $((\alpha,Y),(\alpha',Y'))$, it follows that
		\begin{align*}
			&\|F(Y),F'(Y,Y')\|_{X,p,[s,t]}\\
			&\quad \lesssim \|f\|_{C^2_b} (1 + |(\alpha'_s,Y'_s)| + \|(\alpha',Y')\|_{p,[s,t]} + \|R^{(\alpha,Y)}\|_{\p,[s,t]})^2 (1 + \|X\|_{p,[s,t]})^2 \\
			&\quad \lesssim \|f\|_{C^2_b} (1 + \|\alpha,\alpha'\|_{X,p})^2 (1 + \|Y,Y'\|_{X,p,[s,t]})^2 (1 + \|X\|_{p,[s,t]})^2.
		\end{align*}
		The growth conditions thus hold with constant $C_F = \|f\|_{C^2_b}$ up to a multiplicative constant, which depends on $p$, $\|\alpha\|_p$ and $\|\alpha,\alpha'\|_{X,p}$.
		
		Proceeding as in the proof of Lemma~\ref{lemma: transformation of controlled path}, we can show the Lipschitz conditions, observing that
		\begin{equation*}
			|(\alpha,Y)_s - (\alpha,\tY)_s| + \|(\alpha,Y) - (\alpha,\tY)\|_{p,[s,t]} = |Y_s - \tY_s| + \|Y - \tY\|_{p,[s,t]}
		\end{equation*}
		and
		\begin{equation*}
			\|(\alpha,Y),(\alpha',Y');(\alpha,\tY),(\alpha',\tY')\|_{X,p} = \|Y,Y';\tY,\tY'\|_{X,p},
		\end{equation*}
		and similarly for each summand of the norm, so the Lipschitz conditions hold with constant $C_{F,K,X,X}$, which depends on $p$, $\|f\|_{C^3_b}$, $K$, for $K > 0$, $\|\alpha,\alpha'\|_{X,p}$, and $\|X\|_p$.
	\end{proof}
	
	\begin{proof}[Proof of Corollary~\ref{cor: controlled RDE}]
		\emph{(i)} The existence and uniqueness of the solution follows immediately from Lemma~\ref{lemma: transformation of controlled path with control} and Theorem~\ref{thm: existence and uniqueness}. For the a priori estimate, note that $C_F \lesssim_p \|f\|_{C^2_b} (1 + \|\alpha\|_p + \|\alpha,\alpha'\|_{X,p})^2$.
		
		\emph{(ii)} To apply the continuity result presented in Theorem~\ref{thm: continuity}, we need to ensure that the functionals satisfy the required assumptions. For $(F,F'), (\tF,\tF')$, this is given in Lemma~\ref{lemma: transformation of controlled path with control}. Analogously, since $f - \tf \in C^3_b$, we note that for
		\begin{equation*}
			(Y,Y') \mapsto ((f-\tf)((\alpha,Y)), \D (f-\tf) ((\alpha,Y)) (\alpha',Y'))
		\end{equation*}
		the growth conditions hold with constant equal to $\|f - \tf\|_{C^2_b}$ up to a multiplicative constant which depends on $p$, $\|\alpha\|_p$ and $\|\alpha,\alpha'\|_{X,p}$. Further, it follows from the proofs of~\cite[Lemma~3.1 and Lemma~3.5]{Friz2018} that the growth conditions hold for
		\begin{equation*}
			(Y,Y') \mapsto (f((\alpha,Y)) - f((\tilde{\alpha},Y)), \D f((\alpha,Y)) (\alpha',Y') - \D f((\tilde{\alpha},Y)) (\tilde{\alpha}',Y'))
		\end{equation*}
		with constant equal to $|\alpha_0 - \tilde{\alpha}_0| + \|\alpha - \tilde{\alpha}\|_p + \|\alpha,\alpha';\tilde{\alpha},\tilde{\alpha}'\|_{X,p}$ up to a multiplicative constant which depends on $p$, $\|f\|_{C^3_b}$, $\|\alpha\|_p$, $\|\tilde{\alpha}\|_p$, $\|\alpha,\alpha'\|_{X,p}$, $\|\tilde{\alpha},\tilde{\alpha}'\|_{X,p}$.
		
		This implies that
		\begin{align*}
			&(F-\tF, F'-\tF') \colon \cV^p_X([0,T];\R^k) \to \cV^p_X([0,T];\cL(\R^d;\R^k)), \\
			&((F-\tF)(Y),(F-\tF)'(Y,Y')) \\
			&\quad:= (f((\alpha,Y)) - \tf((\tilde{\alpha},Y)), \D f((\alpha,Y)) (\alpha',Y')- \D \tf((\tilde{\alpha},Y)) (\tilde{\alpha}',Y')),
		\end{align*}
		satisfies the corresponding estimates in Assumption~\ref{assumption: Lipschitz} with
		\begin{equation*}
			C_{F-\tF} = \|f - \tf\|_{C^2_b} + |\alpha_0 - \tilde{\alpha}_0| + \|\alpha - \tilde{\alpha}\|_p + \|\alpha,\alpha';\tilde{\alpha},\tilde{\alpha}'\|_{X,p}
		\end{equation*}
		up to a multiplicative constant, which depends on $p$, $\|f\|_{C^3_b} \vee \|\tf\|_{C^3_b}$, $\|\alpha\|_p$, $\|\tilde{\alpha}\|_p$, $\|\alpha,\alpha'\|_{X,p}$, $\|\tilde{\alpha},\tilde{\alpha}'\|_{X,p}$.
	\end{proof}
	
	\subsection{RDEs with discrete time dependence}
	
	Let us consider the rough differential equation with discrete time dependence
	\begin{equation}\label{eq: RDE with discrete time dependence}
		Y_t = y_t + \int_0^t f(Y_s, Y_{s \wedge r_1}, \ldots Y_{s \wedge r_\ell}) \dd \bX_s, \quad t \in [0,T],
	\end{equation}
	where $\bX \in \cD^p([0,T];\R^d)$ for $p \in (2,3)$, $(y,y') \in \cV^p_X([0,T];\R^k)$, $f \in C^3_b(\R^{k(\ell+1)};\cL(\R^d;\R^k))$, and $r_1 < \dots < r_\ell$ be given time points in $[0,T]$, with $\ell \in \N$. For continuous rough paths as driving signals, the existence (without uniqueness) of a solution to the RDE~\eqref{eq: RDE with discrete time dependence} was proven in~\cite[Example~4.2 and Theorem~4.4]{Ananova2023}. The next proposition provides an existence, uniqueness and continuity result for RDEs with discrete time dependence driven by c{\`a}dl{\`a}g $p$-rough paths.
	
	\medskip
	
	\begin{proposition}~
		\begin{itemize}
			\item[(i)] In the above setting, there exists a unique solution to the RDE with discrete time dependence ~\eqref{eq: RDE with discrete time dependence}. Moreover, there exists a componentwise non-decreasing function $K_p \colon \N \times [0, \infty)^3 \to [0,\infty)$ such that
			\begin{equation*}
				\|Y,Y'\|_{X,p} \leq K_p( \ell, \|f\|_{C^2_b}, \|y,y'\|_{X,p}, \|\bX\|_p).
			\end{equation*}
			\item[(ii)] Let $(Y,Y') \in \cV^p_X([0,T];\R^k)$ be the unique solution to the RDE with time discrete dependence~\eqref{eq: RDE with discrete time dependence}. Moreover, let $(\ty,\ty') \in \cV^p_X([0,T];\R^k)$ with corresponding solution $(\tY,\tY') \in \cV^p_X([0,T];\R^k)$, and suppose that $\|Y,Y'\|_{X,p}, \|\tY,\tY'\|_{X,p} \leq K$, for some ${K > 0}$. Then, we have the estimate
			\begin{equation*}
				|Y_0 - \tY_0| + \|Y,Y';\tY,\tY'\|_{X,p} \lesssim |y_0 - \ty_0| + \|y,y';\ty,\ty'\|_{X,p},
			\end{equation*}
			where the implicit multiplicative constant depends on $p$, $\ell$, $\|f\|_{C^3_b}$, $K$, and $\|\bX\|_p$.
		\end{itemize}
	\end{proposition}
	
	\begin{proof}
		\emph{(i)} On the interval $[0,r_1]$, we extend the vector field $f$ to map into the space of controlled paths by setting
		\begin{equation*}
			(F,F') \colon \cV^p_X([0,r_1];\R^k) \to \cV^p_X([0,r_1]; \cL(\R^d;\R^k)), \quad (F(Y),F'(Y,Y')) := (f(\bar{Y}), \D f(\bar{Y}) \bar{Y}'),
		\end{equation*}
		with $(\bar{Y}, \bar{Y}') = ((Y, Y, \ldots, Y),(Y', Y', \ldots, Y')) \in \cV^p_X([0,T];\R^{k(\ell+1)})$.	It follows analogously to Lemma~\ref{lemma: transformation of controlled path} that the functional satisfies Assumption~\ref{assumption: Lipschitz} (i) and (ii) with constants depending additionally on $\ell$, that is, $C_F = \|f\|_{C^2_b}$ up to a multiplicative constant, which depends on $p$ and $\ell$, and $C_{F,K,X,X}$ depends on $p$, $\ell$, $\|f\|_{C^3_b}$, $K$, for $K > 0$, and $\|X\|_p$. Note that it is sufficient to check Assumption~\ref{assumption: Lipschitz}~(ii) for $X = \tX$. We can thus apply Theorem~\ref{thm: existence and uniqueness} to show that there exists a unique solution to the RDE~\eqref{eq: RDE with discrete time dependence} on the interval $[0,r_1)$. We now aim to solve the RDE~\eqref{eq: RDE with discrete time dependence} iteratively on the subintervals $[r_i,r_{i+1})$, $i = 1, \ldots, \ell$, with $r_{\ell+1} = T$. Given the solution on $[r_{i-1},r_i)$, with $r_0 = 0$, the value $Y_{r_i}$ is determined by the jump of $\bX$ at time $r_i$. We therefore consider $(y_i,y_i') \in \cV^p_X([r_i,r_{i+1}]; \R^k)$, where
		\begin{equation*}
			y_{i;t} = y_t + Y_{r_i-} - y_{r_i-} + F_{r_i-}(Y) X_{r_i-,r_i} + F'_{r_i-}(Y,Y') \X_{r_i-,r_i}, \quad t \in [r_i,r_{i+1}],
		\end{equation*}
		for every $i = 1, \ldots, \ell$, and $(\alpha_i,\alpha_i') \in \cV^p_X([r_i,r_{i+1}];\R^{ik})$ be a fixed controlled path, with $\alpha_{i,t} = (Y_{r_1}, \ldots, Y_{r_i})$, $t \in [r_i,r_{i+1})$. We set
		\begin{align*}
			&(F,F') \colon \cV^p_X([r_i,r_{i+1}];\R^k) \to \cV^p_X([r_i,r_{i+1}];\cL(\R^d;\R^k)), \\
			&(F(Y),F'(Y,Y')) = (f((Y, \alpha_i, Y, \ldots, Y)), \D f((Y, \alpha_i, Y, \ldots, Y)) (Y', \alpha_i', Y', \ldots, Y'))
			&\intertext{for $i = 1, \ldots, \ell - 1$, and}
			&(F(Y),F'(Y,Y')) = (f((Y,\alpha_\ell)), \D f((Y,\alpha_\ell)) (Y', \alpha_\ell')),
		\end{align*}
		for $i = \ell$. Analogously to Lemma~\ref{lemma: transformation of controlled path with control}, we can show that the functional satisfies Assumption~\ref{assumption: Lipschitz}~(i) and (ii) with constants depending additionally on $\ell$, that is $C_F = \|f\|_{C^2_b}$ up to a multiplicative constant, which depends on $p$ and $\ell$, see the definition of $(\alpha_i,\alpha_i') \in \cV^p_X([r_i,r_{i+1}])$, and $C_{F,K,X,X}$ depends on $p$, $\ell$, $\|f\|_{C^3_b}$, $K$, for $K > 0$, and $\|X\|_p$. Note that it is again sufficient to check Assumption~\ref{assumption: Lipschitz}~(ii) for $X = \tX$. We can thus again apply Theorem~\ref{thm: existence and uniqueness} to show that there exists a unique solution to the RDE~\eqref{eq: RDE with discrete time dependence} on the interval $[r_i,r_{i+1})$, that is
		\begin{equation*}
			Y_t = y_{i;t} + \int_{r_i}^t F_s(Y) \dd \bX_s, \quad t \in [r_i,r_{i+1}),
		\end{equation*}
		for every $i = 1, \ldots, \ell$. Then, by pasting the solutions on each of these subintervals together, we obtain a unique global solution $Y$, which holds over the entire interval $[0,T]$.
		
		The a priori estimate follows by iteratively combining the a priori estimate of Corollary~\ref{cor: controlled RDE}, noting that $\alpha_{i,t} = (Y_{r_1}, \ldots, Y_{r_i})$, $t \in [r_i,r_{i+1})$, for $i = 1, \ldots, \ell$.
		
		\emph{(ii)} \emph{Local estimate on $[0,r_1]$.} To apply the continuity result presented in Theorem~\ref{thm: continuity} on the subinterval $[0,r_1]$, we need to ensure that the functionals satisfy the required assumptions.
		
		For $(F,F')$, this is shown in the proof of~(i), and as we aim to obtain continuity of the solution map as a function of the initial condition $(y,y')$, not the vector field $f$, on the interval $[0,r_1]$ we may consider $(F,F') = (\tF,\tF')$, so, $(F-\tF,F'-\tF') = 0$. Theorem~\ref{thm: continuity} now gives that
		\begin{equation*}
			\|Y' - \tY'\|_{p,[0,r_1]} + \|R^Y - R^{\tY}\|_{p,[0,r_1]} \lesssim |y_0 - \ty_0| + \|y,y';\ty,\ty'\|_{X,p,[0,r_1]},
		\end{equation*}
		where the implicit multiplicative constant depends on $p$, $\ell$, $\|f\|_{C^3_b}$, $K$, and $\|\bX\|_p$.
		
		\emph{Local estimate on $[r_i,r_{i+1}]$, $i = 1, \ldots, \ell$.} To apply the continuity result presented in Theorem~\ref{thm: continuity}, we need to ensure that the functionals satisfy the required assumptions. For $(F,F'), (\tF,\tF')$, this is shown in the proof of~(i), and for $(F-\tF,F'-\tF')$, in the proof of Corollary~\ref{cor: controlled RDE}~(ii), where the constant $C_{F-\tF}$ depends additionally on $\ell$. Theorem~\ref{thm: continuity} then implies that
		\begin{equation*}
			\|Y' - \tY'\|_{p,[r_i,r_{i+1}]} + \|R^Y - R^{\tY}\|_{p,[r_i,r_{i+1}]} \lesssim |Y_{r_i} - \tY_{r_i}| + \|y,y';\ty,\ty'\|_{X,p,[r_i,r_{i+1}]},
		\end{equation*}
		where the implicit multiplicative depends on $p$, $\|f\|_{C^3_b}$, $K$ and $\|\bX\|_p$, see the definition of $(\alpha_i,\alpha_i'), (\tilde{\alpha}_i,\tilde{\alpha}_i') \in \cV^p_X([r_i,r_{i+1}];\R^{ik})$, $(y_i,y_i'), (\ty_i,\ty_i') \in \cV^p_X([r_i,r_{i+1}];\R^k)$.
		
		\emph{Global estimate.} Using the methods of the proof of Theorem~\ref{thm: continuity}, and applying the local estimates on the subintervals $[r_i,r_{i+1}]$, $i = 0, 1, \ldots, \ell$, one can then derive that
		\begin{equation*}
			\|Y' - \tY'\|_p + \|R^Y - R^{\tY}\|_\p \lesssim |y_0 - \ty_0| + \|y,y';\ty,\ty'\|_{X,p},
		\end{equation*}
		which implies the estimate.
	\end{proof}
	
	\subsection{RDEs with constant delay}\label{sec: RDEs with constant delay}
	
	Maybe the most prominent example of rough functional differential equations are RDEs with constant delay, cf.~e.g.~\cite{Ferrante2006, Neuenkirch2008, Besalu2014, Chhaibi2020, Besalu2020}. In the present subsection we consider the delayed rough differential equation
	\begin{equation}\label{eq: RDE with constant delay}
		Y_t = y_t + \int_0^t f(Y_s, Y_{s-r_1}, \ldots, Y_{s-r_\ell}) \dd \bX_s, \quad t \in [0,T],
	\end{equation}
	where $\bX \in \cD^p([0,T];\R^d)$ for $p \in (2,3)$, $(y,y') \in \cV^p_X([0,T];\R^k)$, $f \in C^3_b(\R^{k(\ell+1)};\cL(\R^d;\R^k))$, and constant delays $0 < r_1 < \dots < r_\ell$ with $\ell \in \N$. To give a rigorous mathematical meaning to the RDE~\eqref{eq: RDE with constant delay}, we follow the approach of Neuenkirch, Nourdin and Tindel~\cite{Neuenkirch2008}: we assume that the driving rough path~$\bX$ is of the form
	\begin{equation*}
		X_t = (Z_t, Z_{t-r_1}, \ldots, Z_{t-r_\ell}), \quad t \in [0,T],
	\end{equation*}
	for a path $Z \in D^p([-r_\ell,T];\R^e)$ with $d = e(\ell+1)$. We extend the vector field $f$ to map into the space of controlled paths by setting
	\begin{equation*}
		(F(Y),F'(Y,Y')) := (f((Y,\alpha)), \D f((Y,\alpha)) (Y',\alpha')),
	\end{equation*}
	for $(\alpha,\alpha') \in \cV^p_X([0,T];\R^{k\ell})$, where
	\begin{equation}\label{eq: fixed extension}
		\alpha = (\alpha_1, \ldots, \alpha_\ell) \quad \text{with} \quad \alpha_{j,t} :=
		\begin{cases}
			Y_{t-r_j}, \quad &t \in [r_j,T] \\
			Y_{j;t}, \quad &t \in [0,r_j)
		\end{cases},
	\end{equation}
	for fixed controlled paths $(Y_j,Y_j') \in \cV^p_{Z_{\cdot-r_j}}([0,T];\R^k)$ and every $j = 1, \ldots, \ell$. This includes the natural case $Y_j = \xi_{\cdot-r_j}$ for an initial path $\xi \in \cV^p_Z([-r_\ell,T]; \R^k)$.
	
	Note that the postulated form of the rough path $\bX$ is essential to ensure the well-posedness of the rough integral appearing in~\eqref{eq: RDE with constant delay} and the extension of the solution $Y$ to the interval $[-r_\ell,0]$ is a standard and necessary way to give a meaning to $f(Y_s, Y_{s-r_1}, \ldots, Y_{s-r_\ell})$ on the entire interval $[0,T]$.
	
	For delayed RDEs of the form~\eqref{eq: RDE with constant delay} driven by $\alpha$-H\"older continuous rough paths, existence, uniqueness and continuity of the It\^o--Lyons map were first proven in~\cite{Neuenkirch2008} for $\alpha \in (\frac{1}{3},\frac{1}{2})$. These results were extended in~\cite{Tindel2012} to $\alpha$-H{\"o}lder continuous rough paths for $\alpha \in (\frac{1}{4},\frac{1}{3})$. A paracontrolled distribution approach to RDEs with constant delay can be found in~\cite{Promel2021}. Based on the general results of Section~\ref{sec: RFDE}, we can derive the follow proposition.
	
	\begin{proposition}\label{prop: RDE with constant delay}~
		\begin{itemize}
			\item[(i)] In the above setting, there exists a unique solution to the delayed RDE~\eqref{eq: RDE with constant delay}. Moreover, there exists a componentwise non-decreasing function  $K_p \colon \N \times [0, \infty)^4 \to [0,\infty)$ such that
			\begin{equation*}
				\|Y,Y'\|_{X,p} \leq K_p \Big( \ell, \|f\|_{C^2_b}, \|y,y'\|_{X,p}, \sum_{j=1}^{\ell} \|Y_j,Y_j'\|_{Z_{\cdot-r_j},p }, \|\bX\|_p \Big).
			\end{equation*}
			\item[(ii)] Let $(Y,Y') \in \cV^p_X([0,T];\R^k)$ be the unique solution to the rough differential equation with constant delay~\eqref{eq: RDE with constant delay}. Moreover, consider $(\ty,\ty') \in \cV^p_X([0,T];\R^k)$, and fixed controlled paths $(\tY_j,\tY_j') \in \cV^p_{Z_{\cdot-r_j}}([0,T];\R^k)$, $j = 1, \ldots, \ell$, with corresponding solution $(\tY,\tY') \in \cV^p_X([0,T];\R^k)$.
			
			Suppose that $\|Y,Y'\|_{X,p}, \|\tY,\tY'\|_{X,p} \leq K$, for some $K > 0$, and that $\|Y_j\|_p$, $\|\tY_j\|_p$, $\|Y_j,Y_j\|_{X,p}$, $\|\tY_j,\tY_j'\|_{X,p} \leq L$, for some $L > 0$, $j = 1, \ldots, \ell$. Then, we have the estimate
			\begin{align*}
				&|Y_0 - \tY_0| + \|Y,Y';\tY,\tY'\|_{X,p} \\
				&\quad \lesssim |y_0 - \ty_0| + \|y,y';\ty,\ty'\|_{X,p} + \sum_{j=1}^\ell |Y_{j;0} - \tY_{j;0}| + \sum_{j=1}^\ell \|Y_j,Y_j';\tY_j,\tY_j'\|_{Z_{\cdot-r_j},p},
			\end{align*}
			where the implicit multiplicative constant depends on $p$, $\ell$, $r_1$, $T$, $\|f\|_{C^3_b}$, $K$, $L$, and $\|\bX\|_p$.
		\end{itemize}
	\end{proposition}
	
	\begin{proof}
		\emph{(i)} The existence and uniqueness of the solution follows by iteratively applying Corollary~\ref{cor: controlled RDE}~(i) to intervals of the length $r_1$.
		
		More precisely, we consider the functional
		\begin{equation*}
			(F,F') \colon \cV^p_X([0,r_1];\R^k) \to \cV^p_X([0,r_1];\cL(\R^d;\R^k)),
		\end{equation*}
		for $(\alpha,\alpha') \in \cV^p_X([0,r_1];\R^{k\ell})$ given by~\eqref{eq: fixed extension}, and apply Corollary~\ref{cor: controlled RDE}~(i) to show that there exists a unique solution to the RDE~\eqref{eq: RDE with constant delay} on the interval $[0,r_1)$.
		
		We now aim to solve the RDE~\eqref{eq: RDE with constant delay} iteratively on the subintervals $[ir_1,(i+1)r_1]$, $i = 1, \ldots, N-1$, assuming that $T = N r_1$ for some $N \in \N$. Given the solution on $[(i-1)r_1,ir_1)$, the value $Y_{ir_1}$ is determined by the jump of $\bX$ on $ir_1$. We therefore consider $(y_i, y_i') \in \cV^p_X([ir_1,(i+1)r_1];\R^k)$, where
		\begin{equation*}
			y_{i;t} = y_t + Y_{ir_1-} - y_{ir_1-} + F_{ir_1-}(Y) X_{ir_1-,ir_1} + F'_{ir_1-}(Y,Y') \X_{ir_1-,ir_1}, \quad t \in [ir_1,(i+1)r_1],
		\end{equation*}
		for every $i = 1, \ldots, N-1$, and
		\begin{equation*}
			(F,F') \colon \cV^p_X([ir_1,(i+1)r_1];\R^k) \to \cV^p_X([ir_1,(i+1)r_1];\cL(\R^d;\R^k)),
		\end{equation*}
		for $(\alpha,\alpha') \in \cV^p_X([ir_1,(i+1)r_1];\R^{k\ell})$ given by~\eqref{eq: fixed extension}. We again apply Corollary~\ref{cor: controlled RDE}~(i) to show that there exists a unique solution to the RDE~\eqref{eq: RDE with constant delay} on the interval $[ir_1,(i+1)r_1)$, that is
		\begin{equation*}
			Y_t = y_{i;t} + \int_{ir_1}^t F_s(Y) \dd \bX_s, \quad t \in [ir_1,(i+1)r_1)
		\end{equation*}
		for every $i = 1, \ldots, N-1$. Then, by pasting solutions on each of these subintervals together, we obtain a unique global solution~$Y$ to the RDE~\eqref{eq: RDE with constant delay}, which holds over the interval~$[0,T]$.
		
		The a priori estimate follows by iteratively combining the a priori estimate of Corollary~\ref{cor: controlled RDE}, and by the definition of $\alpha$ in~\eqref{eq: fixed extension}.
		
		\emph{(ii)} \emph{Local estimate on $[ir_1,(i+1)r_1]$, $i = 0, \ldots, N-1$.} To apply the continuity result presented in Theorem~\ref{thm: continuity} on the subintervals $[ir_1,(i+1)r_1]$, we need to ensure that the functionals satisfy the required assumptions. This is given for $(F,F'), (\tF,\tF')$ in Lemma~\ref{lemma: transformation of controlled path with control}, and for $(F-\tF,F'-\tF')$ we refer to the proof of Corollary~\ref{cor: controlled RDE}~(ii), and write $C_{F-\tF,i}$ for the corresponding constant. By Theorem~\ref{thm: continuity}, it then holds the estimate
		\begin{align*}
			&\|Y'-\tY'\|_{p,[ir_1,(i+1)r_1]} + \|R^Y - R^{\tY}\|_{\p,[ir_1,(i+1)r_1]} \\
			&\quad \lesssim |Y_{ir_1} - \tY_{ir_1}| + |F_{ir_1}(Y) - \tF_{ir_1}(\tY)| + \|y,y';\ty,\ty'\|_{p,[ir_1,(i+1)r_1]} + C_{F-\tF,i} \\
			&\quad \lesssim |y_0 - \ty_0| + \|y,y';\ty,\ty'\|_{X,p} + \sum_{j=1}^{\ell} |Y_{j;0} - \tY_{j;0}| + \|Y_j,Y_j';\tY_j,\tY_j'\|_{Z_{\cdot-r_j},p} \\
			&\qquad + \|Y' - \tY'\|_{p,[0,ir_1]} + \|R^Y - R^{\tY}\|_{\p,[0,ir_1]},
		\end{align*}
		where the implicit multiplicative constant depends on $p$, $\ell$, $\|f\|_{C^2_b}$, $K$, $L$, and $\|\bX\|_p$, see the definition of $(\alpha,\alpha'), (\tilde{\alpha},\tilde{\alpha}') \in \cV^p_X([0,T];\R^{k\ell})$.
		
		\emph{Global estimate.} Iteratively applying the local estimates and, as before, using that $\|\cdot\|_p \leq N \sum_{i=0}^{N-1} \|\cdot\|_{p,[ir_1,(i+1)r_1]}$, one can then derive the estimate.
	\end{proof}
	
	\subsection{RDEs with variable delay}
	
	Rough differential equations with variable delay represent a slight generalization of RDEs with constant delay. More precisely, let us consider the rough differential equation with variable delay
	\begin{equation}\label{eq: RDE with variable delay}
		Y_t = y_t + \int_0^t f(Y_s, Y_{s-\eta(s)}) \dd \bX_s, \quad t \in [0,T],
	\end{equation}
	where $\bX \in \cD^p([0,T];\R^d)$ for $p \in (2,3)$, $(y,y') \in \cV^p_X([0,T];\R^k)$, $f \in C^3_b(\R^{2k};\cL(\R^d;\R^k))$, and $\eta(\cdot)$ be a bounded continuous function with $\eta(t) \geq \epsilon$, $t \in [0,T]$, for some $\epsilon > 0$, and $\bar{\eta} = \sup\{\eta(t)-t : t \in [0,T]\}$. We assume that the driving rough path $\bX$ is of the form
	\begin{equation*}
		X_t = (Z_t, Z_{t-\eta(t)}), \quad t \in [0,T],
	\end{equation*}
	for a path $Z \in D^p([-\bar{\eta},T];\R^e)$ with $d = 2e$. We extend the vector field $f$ into the space of controlled paths by setting
	\begin{align*}
		&(F,F') \colon \cV^p_X([0,T];\R^k) \to \cV^p_X([0,T];\cL(\R^d;\R^k)),\\
		&(F(Y),F'(Y,Y')) = (f((Y,\alpha)), \D f((Y,\alpha)) (Y',\alpha')),
	\end{align*}
	for $(\alpha,\alpha') \in \cV^p_X([0,T];\R^k)$, where
	\begin{equation*}
		\alpha_t :=
		\begin{cases}
			Y_{t-\eta(t)}, \quad &t \geq \eta(t) \\
			Y_{\eta;t}, \quad &t < \eta(t)
		\end{cases},
	\end{equation*}
	for some fixed controlled path $(Y_\eta,Y_\eta') \in \cV^p_{Z_{\cdot - \eta(\cdot)}}([0,T];\R^k)$.
	
	\begin{corollary}~
		\begin{itemize}
			\item[(i)] In the above setting, there exists a unique solution to the delayed RDE~\eqref{eq: RDE with variable delay}. Moreover, there exists a componentwise non-decreasing function $K_p \colon (0,\infty) \times [0, \infty)^4 \to [0,\infty)$ such that
			\begin{equation*}
				\|Y,Y'\|_{X,p} \leq K_p( \epsilon^{-1}, \|f\|_{C^2_b}, \|y,y'\|_{X,p}, \|Y_{\eta},Y_{\eta}'\|_{X,p}, \|\bX\|_p).
			\end{equation*}
			\item[(ii)] Let $(Y,Y') \in \cV^p_X([0,T];\R^k)$ be the unique solution to the RDE with variable delay~\eqref{eq: RDE with variable delay}. Moreover, consider $(\ty,\ty') \in \cV^p_X([0,T];\R^k)$, and a fixed controlled path $(\tY_\eta,\tY_\eta') \in \cV^p_{Z_{\cdot-\eta(\cdot)}}([0,T];\R^k)$ with corresponding solution $(\tY,\tY') \in \cV^p_X([0,T];\R^k)$.
			
			Suppose that $\|Y,Y'\|_{X,p}$, $\|\tY,\tY'\|_{X,p} \leq K$, for some $K > 0$, and $\|Y_\eta,Y_\eta\|_{X,p}$, \allowbreak $\|\tY_\eta,\tY_\eta'\|_{X,p} \leq L$, for some $L > 0$. Then, we have the estimate
			\begin{align*}
				&|Y_0 - \tY_0| + \|Y,Y';\tY,\tY'\|_{X,p} \\
				&\quad \lesssim |y_0 - \ty_0| + \|y,y';\ty,\ty'\|_{X,p} + |Y_{\eta;0} - \tY_{\eta;0}| + \|Y_\eta,Y_\eta';\tY_\eta,\tY_\eta'\|_{Z_{\cdot-\eta(\cdot)},p},
			\end{align*}
			where the implicit multiplicative constant depends on $p$, $\epsilon$, $T$, $\eta$, $\|f\|_{C^3_b}$, $K$, $L$, and $\|\bX\|_p$.
		\end{itemize}
	\end{corollary}
	
	\begin{proof}
		\emph{(i)} The existence and uniqueness of the solution follows by iteratively applying Corollary~\ref{cor: controlled RDE}~(i) to intervals of the length $\epsilon$, see the proof of Proposition~\ref{prop: RDE with constant delay}~(i). The a priori estimate follows analogously.
		
		\emph{(ii)} The continuity of the solution map follows analogously to Proposition~\ref{prop: RDE with constant delay}~(ii).
	\end{proof}

	\section{Application to stochastic differential equations with delay}\label{sec: rough path lift}
	
	One main application of rough path theory is a pathwise and robust approach to stochastic differential equations, see e.g.~\cite{Friz2020}. In this section we show how a c{\`a}dl{\`a}g martingale and its delayed version can be lifted to a joint random rough path in the spirit of stochastic It{\^o} integration. Consequently, this allows to apply the results on rough functional differential equations, provided in Section~\ref{sec: RFDE}, to It{\^o} stochastic differential equations (SDEs) with constant delay.
	
	\medskip
	
	Throughout the entire section, let us consider constant delays $0 < r_1 < \dots < r_\ell$ with $\ell \in \N$, and let $(\Omega, \mathcal{F},\P)$ be a probability space with a complete and right-continuous filtration $(\mathcal{F}_t)_{t\in[-r_\ell,T]}$. Let $Z = (Z_t)_{t\in[-r_\ell,T]}$ be an $e$-dimensional square-integrable c{\`a}dl{\`a}g martingale that is defined on~$(\Omega, \mathcal{F}, \P$), with $Z_t = 0$ for $t < 0$. The space of all square-integrable random variables on $(\Omega, \mathcal{F},\P)$ is denoted by~$L^2$ and equipped with the standard $L^2$-norm.
	
	\subsection{Delayed martingales as rough paths}
	
	The aim of this subsection is to construct a random rough path above the stochastic process $X=(X_t)_{t\in [0,T]}$, defined as
	\begin{equation*}
		X_t:= (Z_t, Z_{t-r_1}, \ldots, Z_{t-r_\ell}),\qquad t\in [0,T],
	\end{equation*}
	in the spirit of stochastic It{\^o} integration. Recall that, for a martingale $(S_t)_{t\in [0,T]}$ (or, more generally, for a semimartingale $(S_t)_{t\in [0,T]}$), the stochastic It{\^o} integral $\int_0^t \varphi_s \dd S_s$ can be defined whenever $(\varphi_t)_{t\in [0,T]}$ is a stochastic process with left-continuous sample paths with right-limits, which is adapted to the augmented filtration generated by $(S_t)_{t\in [0,T]}$. For a comprehensive introduction to stochastic integration see, e.g., \cite{Protter2005}. In the following, when writing a stochastic integral, like $\int_0^t \varphi_s \dd S_s$, we will always implicitly refer to the augmented filtration generated by $(S_t)_{t\in [0,T]}$ if not explicitly stated otherwise.
	
	To construct a random rough path above the stochastic process $X=(X_t)_{t\in [0,T]}$, the main challenge is to establish the existence of the random integral $\int_0^t Z_{t-r_{j_1}} \dd Z_{t-r_{j_2}}$ for $j_1 < j_2$ since $(Z_{t-r_{j_1}})_{t\in [0,T]}$ is, in general, not adapted to the augmented filtration generated by $(Z_{t-r_{j_2}})_{t\in [0,T]}$.
	
	As a first step to construct a random rough path above the stochastic process~$X$, in the next lemma, we derive the existence of an auxiliary process, inspired by the quadratic co-variation of martingales.
	
	\begin{lemma}\label{lemma: quadratic covariation exists}
		Let $Z = (Z_t)_{t\in[-r_\ell,T]}$ be an $e$-dimensional square-integrable c{\`a}dl{\`a}g martingale that is defined on~$(\Omega, \mathcal{F}, \P$), with $Z_t = 0$ for $t < 0$. Then, for $i_1, i_2 = 1, \ldots, e$, $j_1, j_2 = 0, \ldots, \ell$, $j_1 \neq j_2$, we have
		\begin{equation*}
			\E \bigg[ \sup_{t \in [0,T]} \Big \lvert	\sum_{k=0}^{N_n-1} Z^{i_1}_{t_k^n \wedge t -r_{j_1},t_{k+1}^n \wedge t-r_{j_1}} Z^{i_2}_{t_k^n  \wedge t -r_{j_2},t_{k+1}^n  \wedge t -r_{j_2}} - \sum_{s \leq t} \Delta_s Z^{i_1}_{\cdot-r_{j_1}} \Delta_s Z^{i_2}_{\cdot-r_{j_2}} \Big \rvert^2 \bigg] \longrightarrow 0
		\end{equation*}
		along any sequence of partitions $\mathcal{P}^n = \{0 = t_0^n < t_1^n < \dots < t_{N_n}^n = T\}$, $n \in \N$, of the interval $[0,T]$ with vanishing mesh size, so that $|\mathcal{P}^n| \to 0 $ as $n \to \infty$. Here, we write $\Delta_t H = H_{t-,t}$, with $H_{t-} = \lim_{s\uparrow t} H_s$, for the jump of a stochastic process $H$ at time $t$.
		
		We define the stochastic process
		\begin{equation*}
			[Z^{i_1}_{\cdot-r_{j_1}}, Z^{i_2}_{\cdot-r_{j_2}}]_t := \sum_{s\leq t} \Delta_s Z^{i_1}_{\cdot-r_{j_1}} \Delta_s Z^{i_2}_{\cdot-r_{j_2}}, \quad t \in [0,T].
		\end{equation*}
		This process is c{\`a}dl{\`a}g and has $\P$-almost surely finite $\p$-variation, that is, $[Z^{i_1}_{\cdot-r_{j_1}}, Z^{i_2}_{\cdot-r_{j_2}}]\in D^{\p}([0,T];\R)$ $\P$-almost surely.
	\end{lemma}
	
	\begin{proof}
		We assume~w.l.o.g. that $j_1 = 0$, and write $r = r_{j_2}$, and $M = Z^{i_1}$, $\widetilde{M} = Z^{i_2}$. For $n \in \N$, we define
		\begin{equation*}
			H^n_t := \sum_{k=0}^{N_n-1} \widetilde{M}_{t_k^n-r,t_{k+1}^n-r} \1_{(t_k^n,t_{k+1}^n]}(t), \quad t \in [0,T],
		\end{equation*}
		and note that for $|\mathcal{P}^n| < r$, $H^n$ is indeed a simple predictable process, see~\cite[Chapter~II]{Protter2005}. The It{\^o} integral is then given by
		\begin{equation*}
			\int_0^t H^n_s \dd M_s = \sum_{k=0}^{N_n-1} \widetilde{M}_{t_k^n  \wedge t-r,t_{k+1}^n  \wedge t-r} M_{t_k^n \wedge t,t_{k+1}^n \wedge t}.
		\end{equation*}
		We now aim to show that
		\begin{equation}\label{eq: convergence in H^2}
			\E \bigg[   \int_0^T (H^n_s - H_s)^2 \dd [ M ]_s  \bigg] \to 0, \quad \text{as}\quad n \to \infty,
		\end{equation}
		where $H:= \Delta_{\cdot} \widetilde{M}_{\cdot-r}$, and $[\cdot]$ denotes the quadratic variation. Using the localizing sequence $\tau_m = T \wedge \inf \{t : |\widetilde{M}_t| \geq m \}$, $m \in \N$, and replacing $H^n$ by $H^n_{\cdot \wedge \tau_m}$ and $H$ by $H_{\cdot \wedge \tau_m}$, we may assume that the integrand is uniformly bounded, so that we can apply the dominated convergence theorem. Since $H^n \to H$ converges pointwise as $n\to \infty$, this shows~\eqref{eq: convergence in H^2}.
		
		By~\cite[Chapter~I.4]{Jacod2003}, it follows that
		\begin{equation*}
			\E \bigg[ \sup_{t\in [0,T]} \Big \lvert \int_0^t H^n_s \dd M_s - \int_0^t H_s \dd M_s \Big|^2 \bigg] \to 0
		\end{equation*}
		as $n\to\infty$, thus uniformly in $L^2$, and as
		\begin{equation*}
			\int_0^t H_s \dd M_s = \sum_{s \leq t} \Delta_s \widetilde{M}_{\cdot-r} \Delta_s M,
		\end{equation*}
		this implies the convergence result. Further,
		\begin{equation*}
			4[M,\widetilde{M}_{\cdot-r}] = \sum_{s \leq \cdot} (\Delta_s M + \Delta_s \widetilde{M}_{\cdot-r})^2 - \sum_{s \leq \cdot} (\Delta_s M - \Delta_s \widetilde{M}_{\cdot-r})^2
		\end{equation*}
		has c{\`a}dl{\`a}g sample paths of finite $1$-variation, as both terms on the right hand side are monotonically increasing, which implies that $[M,\widetilde{M}_{\cdot-r}]$ has $\P$-almost surely finite $\p$-variation, and concludes the proof.
	\end{proof}
	
	\begin{proposition}\label{prop: rough path lift martingale}
		Let $p \in (2,3)$, and let $Z = (Z_t)_{t\in[-r,T]}$ be an $e$-dimensional square-integrable c{\`a}dl{\`a}g martingale that is defined on $(\Omega, \mathcal{F}, \P)$, with $Z_t = 0$ for $t < 0$. We set $X = (Z, Z_{\cdot-r_1}, \ldots, Z_{\cdot-r_\ell})$. Then, $X$ can be lifted to a random rough path, by defining $\bX = (X,\X) \in \cD^p([0,T];\R^d)$, $\P$-almost surely, with $d = e(\ell + 1)$, where
		\begin{equation*}
			\X^{ij}_{s,t} := \int_s^t X^i_{u-} \dd X^j_u - X^i_s X^j_{s,t} := \int_0^t X^i_{u-} \dd X^j_u - \int_0^s X^i_{u-} \dd X^j_u - X^i_s X^j_{s,t},
		\end{equation*}
		for $i, j = 1, \ldots, d$ with $i = j$ and $i > j$ such that $X^i = Z^{i_1}_{\cdot-r_{j_1}}$, $X^j = Z^{i_2}_{\cdot-r_{j_2}}$ with $i_1, i_2 = 0, \ldots, e$, $j_1, j_2 = 0, \ldots, \ell$, $j_1 > j_2$, and else
		\begin{equation*}
			\X^{ji}_{s,t} = - \X^{ij}_{s,t} + X^i_{s,t} X^j_{s,t} - [X^i,X^j]_{s,t}
		\end{equation*}
		for any $(s,t) \in \Delta_T$, and where the integration is defined as a stochastic It{\^o} integral, and $[X^i, X^j]$ is defined in Lemma~\ref{lemma: quadratic covariation exists}.
	\end{proposition}
	
	\begin{remark}
		The stochastic integral $\int_0^t Z^{i_1}_{u-r_{j_1}-} \dd Z^{i_2}_{u-r_{j_2}}$ can be defined using classical stochastic It{\^o} integration if $j_1 > j_2$. Indeed, the stochastic process $(Z^{i_2}_{t-r_{j_2}})_{t\in [0,T]}$ is a martingale and the stochastic process $(Z^{i_1}_{t-r_{j_1}})_{t\in [0,T]}$ is predictable, both with respect to the filtration $(\mathcal{F}_{t-r_{j_2}})_{t\in [0,T]}$ with $\mathcal{F}_t:=\{\Omega, \emptyset\}$ for $t<0$.
		
		Further, for $i \geq j$, we have that $\X^{ij}_{s,t} := \int_0^t X^i_{u-} \dd X^j_u - \int_0^s X^i_{u-} \dd X^j_u - X^i_s X^j_{s,t} = \int_s^t X^i_{s,u-} \dd X^j_u \\ = \int_s^t \int_s^{u-} \d X^i_r \dd X^j_u$, that is, $\X^{ij}_{s,t}$ coincides with the $2$-fold iterated integral, for $(s,t) \in \Delta_T$.
	\end{remark}

	\begin{proof}[Proof of Proposition~\ref{prop: rough path lift martingale}]
		First, by definition Chen's relation does hold: Let $0 \leq s \leq v \leq t \leq T$. Then, we have that
		\begin{align*}
			&\X^{ii}_{s,v} + \X^{ii}_{v,t} + X^i_{s,v} X^i_{v,t} \\
			&\quad = \int_s^t X^i_{u-} \dd X^i_u - X^i_s X^i_{s,v} - X^i_v 	X^i_{v,t} + X^i_{s,v} X^i_{v,t} \\
			&\quad = \int_s^t X^i_{u-} \dd X^i_u - X^i_s X^i_{s,t} \\
			&\quad = \X^{ii}_{s,t},
		\end{align*}
		similarly for $\X^{ij}$, and
		\begin{align*}
			&\X^{ji}_{s,v} + \X^{ji}_{v,t} + X^j_{s,v} X^i_{v,t} \\
			&\quad = - \X^{ij}_{s,v} + X^i_{s,v} X^j_{s,v} - [X^i,X^j]_{s,v} - \X^{ij}_{v,t} + X^i_{v,t} X^j_{v,t} - [X^i,X^j]_{v,t} + X^j_{s,v} X^i_{v,t} \\
			&\quad = - \X^{ij}_{s,t} + X^i_{s,v} X^j_{v,t} + X^i_{s,v} X^j_{s,v} + X^i_{v,t} X^j_{v,t} + X^i_{v,t} X^j_{s,v} - [X^i,X^j]_{s,t} \\
			&\quad = - \X^{ij}_{s,t} + X^i_{s,t} X^j_{s,t} - [X^i,X^j]_{s,t} \\
			&\quad = \X^{ji}_{s,t}.
		\end{align*}
		Further, $Z$ has $\P$-almost surely finite $p$-variation, see e.g.~\cite{Lepingle1976}, hence $X \in D^p([0,T];\R^d)$ $\P$-almost surely. Since the maps $s \mapsto \X_{s,t}$ for fixed $t$, and $t \mapsto \X_{s,t}$ for fixed $s$ are both c{\`a}dl{\`a}g, it thus remains to show that $\|\X\|_{\p} < \infty$ $\P$-almost surely.
		
		We define the dyadic stopping times $(\tau_k^n)_{n,k\in\N}$ by
		\begin{equation*}
			t_0^n := 0, \quad \tau_{k+1}^n := \inf \{ t \geq \tau_k^n \, : \, |X_t - X_{\tau_k^n}| \geq 2^{-n}\} \wedge T.
		\end{equation*}
		For $t \in [0,T]$ and $n \in \N$ we introduce the dyadic approximation
		\begin{equation*}
			X_t^n := \sum_{k=0}^\infty X_{\tau_k^n} \1_{(\tau_k^n,\tau_{k+1}^n]}(t) \quad \text{and} \quad \int_0^t X^{i,n}_r \dd X^j_r := \sum_{k=0}^\infty X^i_{\tau_k^n} X^j_{\tau_k^n \wedge t, \tau_{k+1}^n \wedge t},
		\end{equation*}
		for $i = j$ or $i \neq j$ such that $X^i = Z^{i_1}_{\cdot-r_{j_1}}$, $X^j = Z^{i_2}_{\cdot-r_{j_2}}$ for $i_1, i_2 = 1, \ldots, e$, $j_1, j_2 = 0, \ldots, \ell$, $j_1 > j_2$.
		
		We now show that for almost every $\omega \in \Omega$, for every $t \in [0,T]$ and for every $\epsilon \in (0,1)$, there exists a constant $C = C(\omega, \epsilon)$ such that for all $n \in \N$, we have
		\begin{equation}\label{eq: convergence rate martingale}
			\bigg \lvert \bigg( \int_0^t X^{i,n}_u \dd X^j_u - \int_0^t X^i_u \dd X^j_u \bigg)(\omega) \bigg \rvert \leq C 2^{-n(1-\epsilon)}.
		\end{equation}
		Applying the Burkholder--Davis--Gundy inequality, we have that
		\begin{equation*}
			\E \bigg[ \bigg (\sup_{t \in [0,T]} \int_0^t (X^{i,n}_u - X^i_u) \dd X^j_u \bigg )^2 \bigg] \lesssim \E \bigg[\int_0^T (X^{i,n}_u - X^i_u)^2 \dd [X^j]_u \bigg] \lesssim 2^{-2n}, \quad n \in \N,
		\end{equation*}
		where the implicit multiplicative constant depends on the quadratic variation $[X^j]$ of $X^j$. Combining this with Chebyshev's inequality, we obtain for any $\epsilon \in (0,1)$ that
		\begin{equation*}
			\P \bigg( \bigg \lvert \int_0^t (X^{i,n}_u - X^i_u) \dd X^j_u \bigg \rvert \geq 2^{-n(1-\epsilon)} \bigg) \lesssim 2^{2n(1-\epsilon)}2^{-2n} = 2^{-n\epsilon}.
		\end{equation*}
		So by the Borel--Cantelli lemma, we have that
		\begin{equation*}
			\sup_{t \in [0,T]} \Big( \int_0^t X^{i,n}_u \dd X^j_u - \int_0^t X^i_u \dd X^j_u \Big) \lesssim 2^{-n(1-\epsilon)},
		\end{equation*}
		where the implicit multiplicative constant is a random variable which does not depend on $n$, which shows~\eqref{eq: convergence rate martingale}. Proceeding as in the proof of~\cite[Theorem~3.1]{Liu2018}, we can show that $\|\X^{ij}\|_{\p} < \infty$ $\P$-almost surely.
		Further, let $i \neq j$ as above, then we have for any $(s,t) \in \Delta_T$ that
		\begin{equation*}
			|\X_{s,t}^{ji}|^{\p} \lesssim \|\X^{ij}\|^{\p}_{\p,[s,t]} + \|X\|_{p,[s,t]}^p + \|[X^i,X^j]\|^{\p}_{\p,[s,t]}.
		\end{equation*}
		Lemma~\ref{lemma: quadratic covariation exists} then ensures that $\|\X^{ji}\|_{\p} < \infty$ $\P$-almost surely.
	\end{proof}
	
	\begin{remark}
		For a fractional Brownian motion with Hurst index~$H$ and its delayed version a joint rough path was constructed in~\cite{Neuenkirch2008} based on the Russo--Vallois integral~\cite{Russo1993}, assuming that $H > \frac{1}{3}$. This construction was generalized in~\cite{Tindel2012} allowing for $H > \frac{1}{4}$. A related construction of a delayed rough path above a fractional Brownian motion can be found in~\cite{Besalu2014}. For a standard Brownian motion a delayed rough path was also defined in~\cite{Chhaibi2020} based on stochastic It{\^o} integration. While the delayed rough path provided in Proposition~\ref{prop: rough path lift martingale} corresponds to stochastic It{\^o} integration, see Proposition~\ref{prop: RDE and SDE coincide} below, the aforementioned constructions of delayed rough paths above (fractional) Brownian motion correspond to Stratonovich integration.
	\end{remark}
	
	\subsection{SDEs with delay as random RDEs}
	Let us consider the SDE with constant delay
	\begin{equation}\label{eq: SDE with delay}
		\begin{aligned}	Y_t &= y_0 + \int_0^t f(Y_{s-}, Y_{s-r_1-}, \ldots, Y_{s-r_\ell-}) \dd Z_s, \quad t \in [0,T], \\
			Y_t &= y_{t}, \quad t \in [-r_{\ell},0),
		\end{aligned}
	\end{equation}
	where $y \in D^{\p}([-r_ \ell,T];\R^k)$, $f \in C^3_b(\R^{k(\ell+1)};\cL(\R^e;\R^k))$, with $d = e(\ell+1)$, and the integral is defined as a stochastic It{\^o} integral. For a comprehensive introduction to stochastic It{\^o} integration and SDEs we refer, e.g., to the textbook~\cite{Protter2005}. It is known that the SDE~\eqref{eq: SDE with delay} possesses a unique (strong) solution, see e.g.~\cite[Chapter~V, Theorem~7]{Protter2005}. It turns out that the solutions to the SDE~\eqref{eq: SDE with delay} and to the RDE~\eqref{eq: RDE with constant delay} driven by the random rough path $\bX = (X,\X)$, with $\X$ as defined in Proposition~\ref{prop: rough path lift martingale}, coincide $\P$-almost surely.
	
	\begin{proposition}\label{prop: RDE and SDE coincide}
		Let $p \in (2,3)$, and let $Z = (Z_t)_{t\in[-r,T]}$ be an $e$-dimensional square-integrable c{\`a}dl{\`a}g martingale, that is defined on $(\Omega, \mathcal{F}, \P)$, with $Z_t = 0$ for $t < 0$. We set $X := (Z, Z_{\cdot-r_1}, \ldots, Z_{\cdot-r_\ell})$, and let $\bX = (X,\X)$ be the random rough path, with $\X$ defined as in Proposition~\ref{prop: rough path lift martingale}.
		\begin{itemize}
			\item[(i)] Let $(V,V')$ be an adapted stochastic process such that $V$ takes values in $\cL(\R^e;\R^k)$ and, identifying $V$ with its extension (by zero) to take values in $\cL(\R^d;\R^k)$, we assume that $(V(\omega),V'(\omega)) \in \cV^p_{X(\omega)}$ for almost every $\omega \in \Omega$.
			Then, the rough integral exists and coincides $\P$-almost surely with the stochastic It{\^o} integral, that is
			\begin{equation*}
				\int_0^T V_u \dd \bX_u = \int_0^T V_{u-} \dd Z_u, \quad \P \text{-almost surely}.
			\end{equation*}
			\item[(ii)] The solution of the SDE~\eqref{eq: SDE with delay} with coefficient $f$ and driven by $X$, and the solution of the RDE~\eqref{eq: RDE with constant delay} with coefficient $f_0$ and driven by $\bX$, coincide $\P$-almost surely, if $f_0$ be the extension (by zero) of $f$ to a function taking values in $\cL(\R^d;\R^k)$.
		\end{itemize}
	\end{proposition}
	
	\begin{proof}
		\emph{(i)} \emph{Step~1.} \cite[Theorem~31]{Friz2017} gives that
		\begin{equation*}
			\int_0^T V_u \dd \bX_u  = \lim_{|\mathcal{P}|\to 0} \sum_{(s,t) \in \mathcal{P}} (V_s X_{s,t} + V_s' \X_{s,t}),
		\end{equation*}
		where the limit is taken over any sequence of partitions $\mathcal{P}$ of the interval $[0,T]$ with mesh size $|\mathcal{P}|\to 0$,
		and it is known that
		\begin{equation*}
			\sum_{(s,t) \in \mathcal{P}} V_s X_{s,t} \to \int_0^T V_{u-} \dd Z_u,
		\end{equation*}
		in probability as $|\mathcal{P}| \to 0$, noting that $V_s X_{s,t} = V_s Z_{s,t}$ by definition of $V$, see e.g.~\cite[Chapter~II, Theorem~21]{Protter2005}, therefore the convergence holds $\P$-almost surely, possibly along some subsequence.
		
		\emph{Step~2.} We are left to show that
		\begin{equation}\label{eq: compensator term vanishes}
			\lim_{|\mathcal{P}|\to 0} \sum_{(s,t) \in \mathcal{P}} V'_s \X_{s,t} = 0,
		\end{equation}
		$\P$-almost surely, along some subsequence. It suffices to show that
		for $i, j = 1, \dots, d$,
		\begin{equation}
			\label{eq: RDE coincides with SDE}
			\sup_{\tau \in [0, T]}	\Big \lvert \sum_{(s,t) \in \mathcal{P} \cap [0,\tau]} \X^{ij}_{s,t} \Big \rvert \to 0, \qquad \text{as}\qquad |\mathcal{P}|\to 0,
		\end{equation}
		in probability, which then implies $\P$-almost sure convergence, along some subsequence:
		if $V'$ is $\P$-almost surely piecewise constant, then \eqref{eq: RDE coincides with SDE} implies \eqref{eq: compensator term vanishes}. Otherwise, for any $\epsilon > 0$, there exists a suitable piecewise constant approximation ${V'}^{,\epsilon}$ of $V'$ such that
		\begin{equation*}
			\|V' - {V'}^{,\epsilon}\|_\infty \leq \epsilon,
		\end{equation*}
		$\P$-almost surely, see \cite[Proposition~B.1]{Allan2023c}. By a standard interpolation argument (e.g.~\cite[Proposition~5.5]{Friz2010}), it follows, for any $q>p$, that
		\begin{equation*}
			\|V' - {V'}^{,\epsilon}\|_q \leq \|V' - {V'}^{,\epsilon}\|_p^{\frac{p}{q}} \|V' - {V'}^{,\epsilon}\|_\infty^{1 - \frac{p}{q}} \leq C \epsilon^{1-\frac{p}{q}},
		\end{equation*}
		$\P$-almost surely, where the implicit multiplicative constant $C$ is a random variable which does depend only on $p$, $q$ and $\|V'\|_p$. Using \cite[(5.1)]{Young1936}, we obtain that
		\begin{equation*}
			\Big| \sum_{(s,t) \in \mathcal{P}} V'_s \X_{s,t} - \sum_{(s,t) \in \mathcal{P}} {V'}^{,\epsilon}_s \X_{s,t} \Big| \leq \Big(1 + \zeta \Big(\frac{1}{q}+ \frac{2}{p}\Big) \Big) \|V' - {V'}^{,\epsilon}\|_q \|\X\|_{\p} \leq C \epsilon^{1-\frac{p}{q}},
		\end{equation*}
		$\P$-almost surely for any partition $\mathcal{P}$, where the implicit multiplicative constant $C$ is a random variable which does depend only on $p$, $q$, $\|V'\|_p$ and $\|\X\|_{\p}$. Consequently, if
		\begin{equation*}
			\lim_{|\mathcal{P}| \to 0} \sum_{(s,t) \in \mathcal{P}} {V'}^{,\epsilon}_s \X_{s,t} = 0,
		\end{equation*}
		holds $\P$-almost surely, then so does \eqref{eq: compensator term vanishes}, and it suffices to show \eqref{eq: RDE coincides with SDE}.
		
		\emph{Step~3.}
		From here on, for the proof of \eqref{eq: RDE coincides with SDE}, we consider the sequence of partitions $\mathcal{P}^n = \{0 = t^n_0 < t^n_1 < \ldots < t^n_{N_n} = T \}$, $n \in \N$, of the interval $[0,T]$ with vanishing mesh size, so that $|\mathcal{P}^n| \to 0$ as $n \to \infty$. Moreover, let $i \geq j$.
		
		Recall that
		\begin{equation*}
			\X^{ij}_{t^n_k,t^n_{k+1}} = \int_{t^n_k}^{t^n_{k+1}} X_{t^n_k,u-}^i \dd X^j_u  , \qquad  k = 0, \ldots, N_n-1.
		\end{equation*}
		Thus, the Burkholder--Davis--Gundy inequality gives that
		\begin{equation*}
			\E \bigg [ \sup_{\tau \in [0, T]}\Big \lvert \sum_{(s,t) \in \mathcal{P}^n  \cap [0,\tau]} \X^{ij}_{s,t} \Big \rvert^2 \bigg] \lesssim
			\E \bigg [ \int_0^T |X^i_{u_{|\mathcal{P}^n},u-}|^2 \dd [X^j]_u \bigg ],
		\end{equation*}
		where we write $t_{|\mathcal{P}^n} := \max\{t_k^n \in \mathcal{P}^n : t_k^n \leq t\}$, $t \in [0,T]$. Proceeding as in the proof of~\cite[Lemma~6.1]{Friz2023}, one can then show that
		\begin{equation*}
			\E \bigg [ \int_0^T |X^i_{u_{|\mathcal{P}^n},u-}|^2 \dd [X^j]_u \bigg ] \to 0 , \qquad\text{as}\qquad n \to \infty,
		\end{equation*}
		which gives~\eqref{eq: RDE coincides with SDE}.
		
		Therefore, by definition and Lemma~\ref{lemma: quadratic covariation exists} it holds that
		\begin{align*}
			\sum_{k=0}^{N_n-1} \X^{ji}_{t^n_k,t^n_{k+1}} &= - \sum_{k=0}^{N_n-1} \X^{ij}_{t^n_k,t^n_{k+1}} + X^i_{t^n_k,t^n_{k+1}} X^j_{t^n_k,t^n_{k+1}} - [X^i,X^j]_{t^n_k,t^n_{k+1}} \\
			&= - \sum_{k=0}^{N_n-1} \X^{ij}_{t^n_k,t^n_{k+1}} + \sum_{k=0}^{N_n-1} X^i_{t^n_k,t^n_{k+1}} X^j_{t^n_k,t^n_{k+1}} - \sum_{s\leq T} \Delta_s X^i \Delta_s X^j \\
			&\to 0,
		\end{align*}
		as $n\to \infty$, where the convergence holds uniformly in probability, which then concludes the proof.
		
		\emph{(ii)} Let $Y$ be the solution of the rough differential equation~\eqref{eq: RDE with constant delay} driven by the random rough path $\bX = (X,\X)$, see Proposition~\ref{prop: RDE with constant delay}~(i). We note that the assumption on $(V,V')$ in (i) does fit into this setting, where $(V(\omega),V'(\omega)) = (F(Y(\omega)),F'(Y(\omega),Y'(\omega))$ for some functional $F$, see Section~\ref{sec: RDEs with constant delay}. As the rough and It{\^o} integral do coincide $\P$-almost surely by~(i), we infer that $Y$ is also a solution of the SDE~\eqref{eq: SDE with delay}, which has a unique solution (by e.g.~\cite[Chapter~V, Theorem~7]{Protter2005}).
	\end{proof}
	
	\begin{remark}
		As a consequence of Proposition~\ref{prop: RDE with constant delay} and Proposition~\ref{prop: RDE and SDE coincide}~(ii), one can apply the continuity of the It{\^o}--Lyons map (Theorem~\ref{thm: continuity}) to derive pathwise stability results for stochastic differential equations with delay like~\eqref{eq: SDE with delay}. In particular, the map $y\mapsto Y$, mapping the initial path~$y$ to the associated solution~$Y$ of the SDE~\eqref{eq: SDE with delay}, is continuous on the space of controlled paths, which resolves an old observation, pointed out by Mohammed~\cite{Mohammed1986}, about the non-continuity of the flow of stochastic differential equations with delay. The latter is a consequence of the discontinuity of stochastic integration when using an unsuitable topology for the integrands.
	\end{remark}
	
	\begin{remark}
		While we considered square-integrable martingales and the associated stochastic differential equations with constant delay in this section, the presented results can be generalized in a fairly straightforward manner to:
		\begin{itemize}
			\item[(i)] c{\`a}dl{\`a}g local martingales using standard localization arguments;
			\item[(ii)] c{\`a}dl{\`a}g semimartingales using the classical estimates for Young integrals, see e.g. \cite[Proposition~2.4]{Friz2018} and~\cite[Theorem~3.1]{Liu2018}, to show that one can suitably lift~$X$ to a random rough path with additional Young integrals;
			\item[(iii)] Young semimartingales (also known as semimartingales in the sense of Norvai\v{s}a \cite{Norvaivsa2003}), i.e.~$Z = M + \varphi$, for some martingale $M$ and some c{\`a}dl{\`a}g adapted process $\varphi$ with $\varphi(\omega) \in D^q([0,T];\R^e)$ for almost every $\omega \in \Omega$, for some $q \in [1,2)$;
			\item[(iv)] SDEs/RDEs with variable delay of the form~\eqref{eq: RDE with variable delay}, as long as $\eta$ is assumed to be bounded with $\eta(t) \geq \epsilon$, $t \in [0,T]$, for some $\epsilon > 0$.
		\end{itemize}
	\end{remark}

	\appendix
	\section{Local estimates for rough integration}\label{sec: appendix}
	
	The following local estimates are needed to prove the existence and uniqueness result in Theorem~\ref{thm: existence and uniqueness}, and the continuity result in Theorem~\ref{thm: continuity}.
	
	\begin{lemma}\label{lemma: estimate rough integral}
		Let $\bX \in \cD^p([0,T];\R^d)$ for $p \in (2,3)$ and $(Y,Y') \in \cV^p_X([0,T];\R^k)$. Suppose that the non-anticipative functional $(F,F') \colon \cV^p_X([0,T];\R^k) \to \cV^p_X([0,T];\cL(\R^d;\R^k))$ satisfies Assumption~\ref{assumption: existence}~(i) with some constant $C_F$. Then, we have the local estimate
		\begin{equation*}
			\|R^{\int_0^\cdot F(Y) \d \bX}\|_{\p,[s,t]} \lesssim C_F (1 + \|Y,Y'\|_{X,p,[s,t]})^2 (1 + \|X\|_{p,[s,t]})^2 \|\bX\|_{p,[s,t]},
		\end{equation*}
		for all $(s,t) \in \Delta_T$, where the implicit multiplicative constant depends only on $p$.
	\end{lemma}
	
	\begin{proof}
		Let $(V,V') \in \cV^p_X([0,T];\cL(\R^d;\R^k))$, and set $\Xi_{u,v} := V_u X_{u,v} + V'_u \X_{u,v}$ and $\delta \Xi_{u,r,v} := \Xi_{u,v} - \Xi_{u,r} - \Xi_{r,v}$ for $s \leq u < r < v \leq t$. Here, strictly speaking, in writing $V'_u \X_{u,v}$, we use the canonical identification of $\cL(\R^d; \cL(\R^d;\R^k))$ with $\cL(\R^d \otimes \R^d; \R^k)$. We note that
		\begin{align*}
			\Big|R_{u,v}^{\int_0^\cdot V \d \bX}\Big|
			&\leq \Big|\int_u^v V_r \dd \bX_r - \Xi_{u,v} \Big| + |V'_u| |\X_{u,v}| \\
			&\leq \Big| \int_u^v V_r \dd \bX_r - \Xi_{u,v} \Big| + (|V'_s| + \|V'\|_{p,[s,t]}) |\X_{u,v}|.
		\end{align*}
		Using Chen's relation, one can show that
		\begin{equation*}
			- \delta \Xi_{u,r,v} = R_{u,r}^V X_{r,v} + V'_{u,r} \X_{r,v},
		\end{equation*}
		which gives that
		\begin{align*}
			&|\delta \Xi_{u,r,v}| \\
			&\quad \leq \|R^V\|_{\p,[u,r]} \|X\|_{p,[r,v]} + \|V'\|_{p,[u,r]} \|\X\|_{\p,[r,v]} \\
			&\quad = w_{1,1}(u,r)^{\frac{2}{p}} w_{2,1}(r,v)^{\frac{1}{p}} + w_{1,2}(u,r)^{\frac{1}{p}} w_{2,2}(r,v)^{\frac{2}{p}},
		\end{align*}
		where $w_{1,1}(s,t) := \|R^V\|_{\p,[s,t]}^{\p}$, $w_{2,1}(s,t) := \|X\|_{p,[s,t]}^p$, $w_{1,2}(s,t) := \|V'\|_{p,[s,t]}^p$, $w_{2,2}(s,t) := \|\X\|_{\p,[s,t]}^{\p}$, $(s,t) \in \Delta_T$, are control functions and $\frac{1}{p} + \frac{2}{p} > 1$.
		It then follows from the generalized sewing lemma, see~\cite[Theorem~2.5]{Friz2018}, that
		\begin{align*}
			&\|R^{\int_0^\cdot V \d \bX}\|_{\p,[s,t]} \\
			&\quad \lesssim (\|R^V\|_{\p,[s,t]} \|X\|_{p,[s,t]} + \|V'\|_{p,[s,t]} \|\X\|_{\p,[s,t]} + (|V'_s| + \|V'\|_{p,[s,t]}) \|\X\|_{\p,[s,t]}) \notag \\
			&\quad \lesssim \|V,V'\|_{X,p,[s,t]} \|\bX\|_{p,[s,t]},
		\end{align*}
		where the implicit multiplicative constant depends only on $p$. Using Assumption~\ref{assumption: existence} (i), we therefore obtain the estimate for $(V,V') = (F(Y),F'(Y,Y'))$.
	\end{proof}
	
	\begin{lemma}\label{lemma: Lipschitz estimate rough integral}
		For $p \in (2,3)$, suppose $\bX, \tbX \in \cD^p([0,T];\R^d)$, $(Y,Y') \in \cV^p_X([0,T];\R^k)$, $(\tY,\tY')~\in \cV^p_{\tX}([0,T];\R^k)$, and that the non-anticipative functional $(F,F') \colon \cV^p_X([0,T];\R^k) \to \cV^p_X([0,T];\cL(\R^d;\R^k))$ satisfies Assumption~\ref{assumption: Lipschitz} (i) and (ii) given $X,\tX$. Then, we have the local estimate
		\begin{align*}
			&\|R^{\int_0^\cdot F(Y) \d \bX} - R^{\int_0^\cdot F(\tY) \d \tbX}\|_{\p,[s,t]} \\
			&\quad \lesssim \C (|Y_s - \tY_s| + \|Y,Y';\tY,\tY'\|_{X,\tX,p,[s,t]} + \|X - \tX\|_{p,[s,t]}) (\|\bX\|_{p,[s,t]} \vee \|\tbX\|_{p,[s,t]}) \\
			&\qquad + C_F (1 + K)^2 (1 + \|X\|_{p,[s,t]} \vee \|\tX\|_{p,[s,t]})^2 \|\bX;\tbX\|_{p,[s,t]}
		\end{align*}
		for all $(s,t) \in \Delta_T$, if $\|Y,Y'\|_{X,p,[s,t]}$, $\|\tY,\tY'\|_{\tX,p,[s,t]} \leq K$, for some $K > 0$, where the implicit multiplicative constant depends on $p$, $\|\bX\|_p$ and $\|\tbX\|_p$.
	\end{lemma}
	
	\begin{proof}
		It follows from~\cite[Lemma~3.4]{Friz2018} that for any $(V,V') \in \cV^p_X$, $(\widetilde{V},\widetilde{V}') \in \cV^p_{\tX}$,
		\begin{align*}
			&\|R^{\int_0^\cdot V \d \bX} - R^{\int_0^\cdot \widetilde{V} \d \tbX}\|_{\p,[s,t]} \notag \\
			&\quad \lesssim_p (1 + \|\bX\|_{p,[s,t]} + \|\widetilde{\bX}\|_{p,[s,t]}) (\|V,V';\widetilde{V},\widetilde{V}'\|_{X,\tX,p,[s,t]} \|\bX\|_{p,[s,t]} \notag  \\
			&\qquad + \|\widetilde{V},\widetilde{V}'\|_{\tX,p,[s,t]} \|\bX;\tbX\|_{p,[s,t]}) \notag \\
			&\quad \lesssim
			\|V,V';\widetilde{V},\widetilde{V}'\|_{X,\tX,p,[s,t]} \|\bX\|_{p,[s,t]} + \|\widetilde{V},\widetilde{V}'\|_{\tX,p,[s,t]} \|\bX;\tbX\|_{p,[s,t]},
		\end{align*}
		where the implicit multiplicative constant depends on $p$, $\|\bX\|_p$ and $\|\tbX\|_p$. For $(V,V') = (F(Y),F'(Y,Y'))$, $(\widetilde{V},\widetilde{V}') = (F(\tY),F'(\tY,\tY'))$, using Assumption~\ref{assumption: Lipschitz}~(ii), we therefore obtain the estimate.
	\end{proof}
	
	\bibliography{quellen}{}
	\bibliographystyle{amsalpha}
	
\end{document}